\documentclass[11pt]{article}
\usepackage{amsmath,amsthm,amsfonts,amssymb,amscd}
\usepackage[scr]{rsfso}
\usepackage{stackrel}
 \usepackage{amsmath}
\usepackage{enumerate} 
\usepackage{color}
\usepackage{float}
\usepackage{soul}
\usepackage{graphicx}
\usepackage{mathpazo}
\usepackage{fancyhdr}
\usepackage{empheq}
\usepackage[margin=1in]{geometry}
\def\disp{\displaystyle}

\def\XX{\mathbb{X}}
\def\YY{\mathbb{Y}}

\def\Limsup{\mathop{{\rm Lim}\,{\rm sup}}}

\def\Diag{\mbox {\rm Diag}\,}

\def\tto{\rightrightarrows}
\def\Hat{\widehat}
\def\hat{\widehat}
\def\Tilde{\widetilde}
\def\Bar{\overline}
\def\ra{\rangle}
\def\la{\langle}
\def\ve{\varepsilon}
\def\B{\mathbb{B}}
\def\h{\hfill\Box}
\def\R{\mathbb{R}}
\def\ox{\bar{x}}
\def\OX{\Bar{X}}
\def\oy{\bar{y}}
\def\OY{\Bar{Y}}

\def\ov{\bar{v}}
\def\OV{\Bar{V}}

\def\OU{\Bar{U}}

\def\gph{\mbox{\rm gph}\,}

\def\span{\mbox{\rm span}\,}

\def\dom{\mbox{\rm dom}\,}
\def\Ker{\mbox{\rm Ker}\,}

\def\cl*co{\mbox{\rm cl}^*\mbox{\rm co}\,}
\newcommand{\eqdef}{\stackrel{\text{\rm\tiny def}}{=}}
\def\cl{\mbox{\rm cl}\,}

\def\h{\hfill\triangle}
\def\dn{\downarrow}

\def\ph{\varphi}

\def\st{\stackrel}
\def\oR{\Bar{\R}}

\def\lm{\lambda}
\def\gg{\gamma}

\def\al{\alpha}

\def\Lm{\Lambda}

\def\hs7{\hspace*{7pt}}

\def\Id{\mathbb{I}}

\renewcommand{\theequation}{\thesection.\arabic{equation}}

\def\h{\hfill\Box}
\def\kk{\kappa}
\begin{document}

\newtheorem{Theorem}{Theorem}[section]
\newtheorem{Conjecture}[Theorem]{Conjecture}
\newtheorem{Proposition}[Theorem]{Proposition}
\newtheorem{Remark}[Theorem]{Remark}
\newtheorem{Lemma}[Theorem]{Lemma}
\newtheorem{Corollary}[Theorem]{Corollary}
\newtheorem{Definition}[Theorem]{Definition}
\newtheorem{Example}[Theorem]{Example}
\newtheorem{Fact}[Theorem]{Fact}
\newtheorem*{pf}{Proof}
\renewcommand{\theequation}{\thesection.\arabic{equation}}
\normalsize
\normalfont
\medskip
\def\endproof{$\h$\vspace*{0.1in}}

\title{\bf Geometric characterizations of Lipschitz stability for convex optimization problems}
\date{}
\author{TRAN T. A. NGHIA\footnote{Department of Mathematics and Statistics, Oakland University, Rochester, MI 48309, USA; email: nttran@oakland.edu} }

\maketitle
\vspace{-0.2in}

{\small \noindent {\bf Abstract.} In this paper, we mainly study tilt stability and Lipschitz stability of convex optimization problems. Our characterizations are geometric and fully computable in many important cases. As a result, we apply our theory to the group Lasso problem and the nuclear norm minimization problem and reveal that the Lipschitz stability of the solution mapping in these problems is automatic whenever the solution mapping is single-valued.
}
\vspace{0.1in}

\vspace{0.1in}


\section{Introduction}
\setcounter{equation}{0}
Recovering  a signal $x_0\in \XX$ from a noised observation $b\approx \Phi x_0$ with $\Phi:\XX\to \YY$ being a linear operator between two Euclidean spaces is an important problem in many different areas of science and engineering. One of the most popular methods for recovering  approximate solutions of $x_0$ is to solve the following regularized optimization  problem 
\begin{equation}\label{CP}
    \min_{x\in \XX}\qquad \frac{1}{2\mu}\|\Phi x-b\|^2+ g(x),
\end{equation}
where $\mu>0$ is a regularization (tuning) parameter and $g:\XX\to \R$ is a convex function (regularizer). The choice of  regularizer $g$ usually depends on desirable properties of the original signal $x_0$. As $b$ is  also considered a parameter, it is natural to question stability and sensitivity  of the solution mapping of problem~\eqref{CP} with the two parameters $b$ and $\mu$ \cite{MY12,T13,VDFPD17,BLPS21,BBH22,BPS22}. In this paper, we mainly study the Lipschitz stability of the solution mapping  $S:\YY\times \R_{++}\to \XX$ 
\begin{equation}\label{Sp}
S(b,\mu)\eqdef{\rm argmin}\,\left\{\frac{1}{2\mu}\|\Phi x-b\|^2+ g(x)|\; x\in \XX\right\}.
\end{equation}
Lipschitz stability of this mapping was studied recently in \cite{BBH22} for a particular class of $g$ as the $\ell_1$ norm. Since the $\ell_1$ norm is a polyhedral function, \cite{BBH22} is able to employ some known results on Lipschitz stability \cite{M1,RW98} and available second-order computations on the $\ell_1$ norm to obtain sufficient conditions for Lipschitz stability of  $S$ and derive explicit formula for its Lipschitz modulus. However, these second-order computations could be very complicated for nonpolyhedral regularizers such as the {\em nuclear norm}. In the recent papers \cite{BLPS21,BPS22}, sensitivity analysis of solution mapping~\eqref{Sp} was studied  via the so-called {\em conservative Jacobian}, which also needs $S$ to be Lipschitz continuous at first.  We approach Lipschitz stability differently via  the {\em generalized implicit function theorem} \cite{R80,DR09} and {\em tilt stability} \cite{PR98} together with  the  recent geometric characterizations  of {\em strong minima} in \cite{FNP23} without computing complex second-order structures on $g$.  Our ideas are summarized as follows.

The solution mapping \eqref{Sp} is viewed as a {\em generalized equation} \cite{R80} 
\begin{equation}\label{GE}
S(b,\mu)=\left\{x\in \XX|\; 0\in \frac{1}{\mu}\Phi^*(\Phi x-b)+\partial g(x)\right\},
\end{equation}
where $\partial g$ is the {\em convex subdifferential mapping} of $g$.
It is well-known from the notable Robinson implicit function theorem \cite{R80,DR09} that $S$ has a {\em Lipschitz continuous single-valued localization} around $(\bar b,\bar \mu)\in \YY\times \R_{++}$ for $\ox\in S(\bar b,\bar \mu)$ if the following mapping
\begin{equation}\label{Lin}
\Hat S(v)\eqdef \left\{x\in \XX|\; v\in \frac{1}{\bar \mu} \Phi^*(\Phi x-\bar b)+\partial g(x)\right\} \quad \mbox{for}\quad v\in \XX^*,
\end{equation}
has a Lipschitz continuous single-valued localization around $0$ for  $\ox$. This idea dates back to the original work of Robinson \cite{R80} similar to the connection between the classical implicit function theorem and the classical inverse function theorem. Studying Lipschitz stability of  $\Hat S$ in \eqref{Lin} is easier than that of $S$ in \eqref{GE}. One important observation in our study is that Lipschitz stability of $\Hat S$ around $0$ for $\ox$ is equivalent to tilt stability at  $\ox$ of the below problem 
\begin{equation}\label{CP2}
    \min_{x\in \XX}\qquad \frac{1}{2\bar\mu}\|\Phi x-\bar b\|^2+ g(x).
\end{equation}
Tilt stability occurs when  the solution mapping of an optimization problem is Lipschitzian stable as it is shifted by a small tilt (linear) perturbation. Indeed, $\Hat S$ in \eqref{Lin} is the solution mapping of the following  problem
\begin{equation}\label{ST}
    \min_{x\in \XX}\qquad \frac{1}{2\bar\mu}\|\Phi x-\bar b\|^2+ g(x)-\la v,x\ra\quad \mbox{for}\quad v\in \XX^*.
\end{equation}
 Tilt stability was introduced in the landmark paper of Poliquin and Rockafellar \cite{PR98} and developed in different settings \cite{BGM19,CHN18,DMN14, DI15, G24, GO22, GM15, LZ13, OM23, MR12, MN15}. The original result in \cite[Theorem~1.3]{PR98} shows that tilt stability is equivalent to the positive definiteness of the {\em generalized Hessian/second-order limiting subdifferential} of the cost function defined in \cite{M92}. For example, when $g(x)$ is the $\ell_1$-norm, by using the explicit formula of the generalized Hessian of  $\ell_1$ norm in $\R^n$, e.g., \cite[Proposition~8.1]{KMP22}, one can  obtain a full characterization of tilt stability for problem \eqref{CP2}, which turns out to be  equivalent to Tibshirani's  condition \cite[Lemma~1]{T13} for  the corresponding $\ell_1$ optimization problem (a.k.a. Lasso). Consequently,  Tibshirani's condition is shown to be sufficient for Lipschitz stability of solution mapping \eqref{Sp} of Lasso problem as  in  \cite{BBH22} via a different approach. One of the main disadvantages of using  the generalized Hessian is its hard computation in many nonpolyhedral structures. Lots of research papers tried to overcome that obstacle by different ways \cite{AG08,CHN18,DL13,DMN14,GO22,OM23,MN15}.  In this paper, we obtain  a new characterization of tilt stability via {\em second subderivatives} \cite{BS00,RW98}, which seems to be the simplest second order structure for nonsmooth functions.  Our approach relies on the primal-dual characterization of tilt stability \cite[Theorem~3.3]{CHN18} that is also used in \cite{BGM19} to establish explicit conditions for tilt stability in (nonpolyhedral) {\em second-order cone programming} problems.

A significant class of problem \eqref{CP} is the {\em nuclear norm minimization problem}:     
\begin{equation}\label{NNM0}
    \min_{X\in \R^{n_1\times n_2}}\quad \frac{1}{2\mu}\|\Phi X-B\|^2+ \|X\|_*,
\end{equation}
where $\Phi:\R^{n_1\times n_2}\to \R^m$ is a linear operator, $B\in \R^m$, $\mu>0$, and $\|X\|_*$ is the nuclear norm of matrix $X$ being the sum of all its singular values. This problem has high impact in several important research areas such as {\em matrix completion},  {\em multivariate regression}, {\em matrix compressed sensing} \cite{CRPW12,CP10,CR09,W19}.  Unlike the $\ell_1$ norm, the generalized Hessian  of the nuclear norm is very hard to calculate. One of our main motivations  is to characterize tilt stability of problem \eqref{NNM0} and Lipschitz stability of the  solution mapping \eqref{Sp}  without computing the generalized Hessian.

{\bf Our contributions.} As discussed above, tilt stability plays the central role in our paper. More specifically, we study tilt stability of the following problem 
\begin{equation}\label{CPS}
\min_{x\in \mathbb{X}}\quad f(x)+g(x),
\end{equation}
where $f:\XX\to \R$ is a twice continuously differentiable convex function and $g:\XX\to \oR$ is a proper l.s.c. convex function. This is the general model of problem ~\eqref{CP2} when the quadratic loss function is replaced by the function  $f$. Characterization of tilt stability for this problem is important for its own sake due to recent developments of employing  tilt stability to design nonsmooth Newton's method to solve many major optimization problems in the format of \eqref{CPS}; see, e.g., \cite{KMP22,G24,GO21,GO22,MS21}.    One of our main results is to establish a new geometric characterization of tilt stability with two additional conditions on the function $g$: (a) $g$ satisfies the so-called {\em quadratic growth condition for  a given set $\mathcal{M}\subset \XX^*$} and (b) $\partial g$ has {\em relative approximations}  by $\mathcal{M}$. {\em Convex piecewise linear-quadratic functions} \cite{RW98}, including the $\ell_1$ norm and convex picecewise linear functions satisfy these two conditions for  the whole dual space $\XX^*$ due to polyhedral features in their constructions. We also show that the two conditions are valid in other  substantial  nonpolyhedral regularizers such as the $\ell_1/\ell_2$ norm (a.k.a the group Lasso regularizer \cite{YL06}), the nuclear norm, and the indicator function to the set of positive semi-definite matrices. Constructing the set $\mathcal{M}$ in these cases is not trivial, especially for the nuclear norm. 

Another crucial part of our paper is to investigate the Lipschitz stability of the optimal solution mapping $S$ in \eqref{Sp}. In the case of Lasso problem, the recent paper \cite{BBH22} shows that Tibshirani's condition \cite{T13} is sufficient for the latter property. Our geometric characterization of tilt stability for problem \eqref{CP2} are also sufficient for the  Lipschitz stability of $S$ as discussed above. In this paper, we prove that it is also necessary when the function $g$ is either the $\ell_1/\ell_2$ norm (including the case of $\ell_1$ norm) or the nuclear norm. Moreover, for the last two important regularizers,  we show that if the solution mapping~\eqref{Sp} is single-valued, the Lipschitz continuity follows automatically. This is quite a surprising phenomena, as Lipschitz continuity usually does not occur freely. Consequently, Tibshirani's sufficient condition \cite{T13} is also necessary for single-valueness and Lipschitz stability of solution mapping of the Lasso problem. 


 \section{Preliminaries}
\setcounter{equation}{0}
Throughout this paper we suppose that $\XX$ is an Euclidean finite dimensional space with the norm $\|\cdot\|$. Its dual space $\XX^*$ is endowed by the inner product $\la \cdot,\cdot \ra$. $\B_\ve(x)$ stands for the open ball with radius $\ve>0$ and center $x$. The most important geometric structure used in our paper is the (Bouligand/Severi) {\em tangent/contingent cone} \cite[Definition~6.1]{RW98} to a closed set $\Omega\subset \XX$ at a point $\ox\in \Omega$ defined by
\begin{equation}\label{Tg}
    T_\Omega(\ox)\eqdef\{w\in \XX|\; \exists t_k\dn 0, w_k\to w, \ox+t_k w_k\in \Omega\}=\Limsup_{t\dn 0}\frac{\Omega-\ox}{t}. 
\end{equation} 

In the convex case, the polar cone of the tangent cone is known as the {\em normal cone} to $\Omega$ at $\ox$:
\begin{equation}\label{Nor}
N_\Omega(\ox)\eqdef\left[T_\Omega(\ox)\right]^\circ=\{v\in \XX^*|\; \la v,x-\ox\ra\le 0, \forall x\in \Omega\}.
\end{equation}

Let $\varphi:\XX\to \oR\eqdef\R\cup\{+\infty\}$ be a proper lower semi-continuous (l.s.c.) convex function with nonempty domain $\dom \varphi\eqdef\{x\in \XX|\; \varphi(x)< \infty\}$. The  convex subdifferential of $\varphi$ at $\ox\in \dom \varphi$ is 
\[
\partial \varphi(\ox)\eqdef\{v\in \XX^*|\; \varphi(x)-\varphi(\ox)\ge \la v,x-\ox\ra, \forall x\in \XX\}.
\]
The normal cone $N_\Omega(\ox)$ defined above is indeed the subdifferential of the {\em indicator function} $\delta_\Omega(\cdot)$ at $\ox$, where  $\delta_\Omega(x)$ is equal to $0$ if $x\in \Omega$ and equal to $\infty$ otherwise. 

The Fenchel conjugate of $\varphi$ is the function $\varphi^*:\XX^*\to \oR$ defined by
\begin{equation}\label{Fe}
\varphi^*(v)\eqdef\sup\{\la v,x\ra-\varphi(x)|\; x\in \XX\}\quad \mbox{for}\quad v\in \XX^*. 
\end{equation}
It is well-known from the Fenchel-Young's identity  that 
\begin{equation}\label{FY}
v\in \partial \varphi(\ox)\quad \mbox{if and only if}\quad \varphi(\ox)+\varphi^*(v)=\la v,\ox\ra. 
\end{equation}
 Note further that $\partial \varphi:\XX\tto\XX^*$ is a set-valued mapping whose graph is denoted by 
\[
\gph \partial \varphi\eqdef\{(u,v)\in X\times \XX^*|\; v\in \partial \varphi(u)\}. 
\]
Next, let us consider the following convex optimization problem 
\begin{equation}\label{Phi}
  \min_{x\in \XX}\quad \varphi(x).   
\end{equation}
The point $\ox$ is an optimal solution (minimizer) of this problem if and only if $0\in \partial \varphi(\ox)$. Moreover, it is called  a strong solution if there exist some $\ve>0$ and modulus $\kk>0$ such that 
\[
\varphi(x)-\varphi(\ox)\ge \frac{\kk}{2}\|x-\ox\|^2\quad \mbox{for all}\quad x\in \B_\ve(\ox). 
\]
One of the simplest ways to characterize strong minima is using the {\em second subderivatives}  \cite{BS00,RW98}. 
\begin{Definition}[Second subderivative] \label{SS}
The {\em second subderivative} of $\ph$ at $\ox$ for $\ov\in \partial \varphi(\ox)$ is the function $d^2 \ph(\ox|\ov):\XX\to \oR$  defined by 
\begin{equation}\label{Subd}
d^2\ph(\ox|\ov)(w)\eqdef\liminf_{t\dn 0, w^\prime \to w}\dfrac{\ph(\ox+tw^\prime)-\ph(\ox)-t\la \ov, w^\prime\ra}{\frac{1}{2}t^2}\quad \mbox{for} \quad w\in \XX.
\end{equation}
\end{Definition}
It is well-known \cite[Theorem~13.24]{RW98} that $\ox$ is a strong minimizer of $\varphi$ if and only if $0\in \partial \varphi(\ox)$ and there exists $\ell>0$ such that 
\begin{equation}\label{Stro}
    d^2\varphi(\ox|0)(w)\ge\ell\|w\|^2\quad \mbox{for all}\quad w\in \XX. 
\end{equation}
As discussed in the Introduction, tilt stability introduced by Poliquin and Rockafellar \cite[Definition~1.1]{PR98} plays a central role in our paper. We recall the definition here.

\begin{Definition}[Tilt stability]\label{TS} The point $\ox$ is called a tilt-stable solution of the function $\varphi$ (or problem \eqref{Phi}) if there exists $\gg>0$ such that the following optimal solution mapping
\begin{equation}\label{Mg}
M_\gg(v)\eqdef{\rm argmin}\; \{\varphi(x)-\la v,x\ra|\; x\in \B_\gg(\ox)\}\quad \mbox{for}\quad v\in \XX^*
\end{equation}
is single-valued and Lipschitz continuous on some neighborhood of $0\in \XX^*$ with $M_\gg(0)=\ox$. 
\end{Definition}

 Tilt stability in the above definition is designed for general (possibly nonconvex) cost function $\varphi$ as in \cite{PR98}. In this paper, we focus on the case when $\varphi$ is a convex function. As local minimizers are also  global minimizers in convex programming, it is possible to remove the ball $\B_\gg(\ox)$ in \eqref{Mg}. Indeed, $\ox$ is a tilt stable minimizer of the convex problem \eqref{Phi} if and only if the solution mapping
\begin{equation}\label{SM}
    M(v)={\rm argmin}\; \{\varphi(x)-\la v,x\ra|\; x\in \XX\}\quad \mbox{for}\quad v\in \XX^*
\end{equation}
is single-valued and Lipschitz continuous around $0\in \XX^*$ with $M(0)=\ox$. Note that 
\[
M(v)=(\partial \varphi)^{-1}(v)=\partial\varphi^*(v)
\]
is a convex set for any $v\in \XX^*$. Thus $\ox$ is a tilt-stable solution of problem \eqref{Phi} if and only if $\partial \varphi$ is {\em strongly metrically regular} at $0$ for $\ox$ in the sense that  $(\partial \varphi)^{-1}(0)=\{\ox\}$ and the inverse mapping $(\partial \varphi)^{-1}$ is Lipschitz continuous around $0$; see \cite{AG08, DL13, DMN14, LZ13,MN15,ZT95}. Note that this definition of strong metric regularity is slightly different from the traditional one in \cite[Section~3G]{DR09}, which needs a {\em localization}. In particular, \cite[Section~3G]{DR09} says $\partial \varphi$ is  strongly metrically regular at $0$ for $\ox$ if $(\partial \varphi)^{-1}$ has a {\em Lipschitz continuous singe-valued localization} $s$ around $0$ for $\ox$ in the sense that there exists a neighborhood $U\times V\subset \XX\times \XX^*$ of $(\ox,0)$ such that $\gph s=\gph (\partial \varphi)^{-1}\cap (V\times U)$ and that $s$ is single-valued and Lipschitz continuous in $V$ with $s(0)=\ox$. In this case, for any $v\in V$, $s(v)$ is an isolated point of $(\partial\varphi)^{-1}(v)$. Since  $(\partial\varphi)^{-1}(v)$ is a convex set, the latter implies that $s(v)=(\partial\varphi)^{-1}(v)$, i.e., $(\partial\varphi)^{-1}(\cdot)$ is single-valued and Lipschitz continuous on $V$. 

There are  different characterizations of tilt stability. In the original paper \cite{PR98}, Poliquin and Rockafellar used the {\em generalized Hessian} of the function $\varphi$, which is the {\em limiting coderivative} of the subdifferential $\partial \varphi$ introduced by Mordukhovich \cite{M92} to characterize tilt stability. The disadvantage of this construction is that it may be very hard to compute. In this paper, we avoid calculating it  and actually employ the second subderivative in Definition~\ref{SS}. Before going there, we recall  two  known characterizations of tilt stability that are the {\em uniform quadratic growth condition} \cite{AG08,DL13, DMN14} and the positive definiteness of the {\em subgradient graphical derivative} \cite{CHN18}, a primal-dual structure that is also used in \cite{BGM19} to derive explicit conditions for tilt stability of (nonpolyhedral) second-order cone programming problems.

\begin{Theorem}[Characterization of tilt stability via subgradient]\label{CharT} Suppose that $0\in \partial \varphi(\ox)$.  The following are equivalent:

\begin{itemize}
    \item[{\bf (i)}] {\rm (Tilt stability)} The point $\ox$ is a tilt-stable solution of problem \eqref{Phi}.

    \item[{\bf (ii)}] {\rm (Uniform quadratic growth condition)} There exist some $\kk>0$ and an open neighborhood $U\times V\subset \XX\times \XX^*$ of $(\ox,0)$ such that  for any $(u,v)\in \gph \partial \varphi\cap(U\times V)$ one has 
    \begin{equation}\label{UQGC}
\varphi(x)-\varphi(u)-\la v,x-u\ra\ge \frac{\kk}{2}\|x-u\|^2\quad \mbox{for all}\quad x\in U.
    \end{equation}

    \item[{\bf (iii)}] {\rm (Positive definiteness via the subgradient graphical derivative)} There exist $\kk>0$ and an open neighborhood $U\times V\subset \XX\times \XX^*$ of $(\ox,0)$ such that  
\begin{equation}\label{PD}
\la z,w\ra\ge \frac{1}{\kk}\|w\|^2\quad \mbox{whenever} \quad z\in D\partial \varphi(u|v)(w), (u,v)\in {\rm gph} \partial \varphi\cap(U\times V), 
\end{equation}
where $D\partial \varphi(u|v)(\cdot)$ is known as the {\em subgradient graphical derivative} defined by
\begin{equation}\label{}
    D\partial \varphi(u|v)(w)=\{z\in \XX|\; (z,w)\in T_{{\rm gph}\, \partial \varphi}(u,v)\}
\end{equation}
with $T_{{\rm gph}\, \partial \varphi}(u,v)$ being the contingent cone to $\gph\partial\varphi$ at $(u,v)\in \gph \partial \varphi$ defined in \eqref{Tg}. 
\end{itemize}
\end{Theorem}

Next we connect the second subderivative with  the subgradient graphical derivative.

\begin{Lemma}\label{Prop1} Let $\varphi$ be a proper l.s.c. convex function. For any $(u,v)\in \gph \partial \varphi$, we have 
\begin{equation}\label{Inq}
2\la z,w\ra\ge d^2 \varphi(u|v)(w)+d^2\varphi^*(v|u)(z)\quad \mbox{whenever}\quad z\in D\partial \varphi(u|v)(w). 
\end{equation}
Consequently,  
\begin{equation}\label{Inq2}
\la z,w\ra\ge\frac{1}{2} d^2 \varphi(u|v)(w)\quad \mbox{whenever}\quad z\in D\partial \varphi(u|v)(w). 
\end{equation}
\end{Lemma}
\begin{proof} Pick any $z\in D\partial \varphi(u|v)(w)$, we find  sequences $(w_k,z_k)\to (w,z)$ and $t_k\dn 0$ such that 
\begin{equation*}
v+t_kz_k\in \partial \varphi(u+t_k w_k). 
\end{equation*}
It follows from the Fenchel-Young's identity \eqref{FY} that 
\[\begin{array}{ll}
\varphi^*(v+t_kz_k)&=\la v+t_kz_k,u+t_k w_k\ra -\varphi(u+t_k w_k)\\
&=\la v,u\ra+t_k\la v, w_k\ra+t_k\la z_k,u\ra+t_k^2\la z_k,w_k\ra-\varphi(u+t_k w_k)\\
&=\varphi^*(v)+\varphi(u)+t_k\la v, w_k\ra+t_k\la z_k,u\ra+t_k^2\la z_k,w_k\ra-\varphi(u+t_k w_k).
\end{array}\]
Hence we have 
\[
2\la z_k, w_k\ra =\dfrac{\varphi(u+t_k w_k)-\varphi(u)-t_k\la v,w_k\ra}{\frac{1}{2}t_k^2}+\dfrac{\varphi^*(v+t_kz_k)-\varphi^*(v)-t_k\la u,z_k\ra}{\frac{1}{2}t_k^2}\\
\]
By taking the liminf both sides, we obtain  \eqref{Inq}. Moreover, \eqref{Inq2} follows directly from \eqref{Inq}, as $\varphi^*$ is a convex function. 
\end{proof}

When the function $\varphi$ is {\em twice epi-differentiable}  at $u\in \XX$ for $v\in \partial \varphi(u)$ in the sense of \cite[Definition~13.6]{RW98}, it is proved  in \cite[Lemma~3.6 and Theorem~3.7]{CHNT21} that \eqref{Inq2} can be improved to 
\begin{equation}\label{ZW}
\la z,w\ra=d^2 \varphi(u|v)(w)\quad \mbox{whenever}\quad z\in D\partial \varphi(u|v)(w). 
\end{equation}
We are ready to set up a new characterization of tilt stability via second subderivative as below. 

\begin{Corollary}[Characterization of tilt stability via second subderivative]\label{Tilt2} Suppose that $\ox$ is a minimizer of $\varphi$. Then $\ox$ is a tilt-stable solution of problem~\eqref{Phi} if and only if there exist $\ell>0$ and an open neighborhood $U\times V\subset \XX\times \XX^*$ of $(\ox,0)$ such that 
\begin{equation}\label{SubT}
d^2\varphi (u|v)(w)\ge \ell\|w\|^2\quad \mbox{whenever}\quad (u,v)\in \gph \partial \varphi \cap (U\times V)\quad \mbox{and}\quad w\in \XX. 
\end{equation}
\end{Corollary}
\begin{proof} Suppose that $\ox$ is a tilt-stable solution of problem~\eqref{Phi}. By Theorem~\ref{CharT}, the uniform quadratic growth condition \eqref{UQGC} holds with some modulus $\kk>0$ and open neighborhood $U\times V$ of $(\ox,0)$. For any $(u,v)\in \gph \partial \varphi\cap(U\times V)$, note from \eqref{UQGC} that $u$
 is a strong solution with modulus $\kk$ of the function $\varphi(\cdot)-\la v, \cdot\ra$. It follows from \eqref{Stro} that \eqref{SubT} is satisfied.

To prove the converse implication, suppose that inequality \eqref{SubT} holds for some $\ell>0$ and neighborhood $U\times V$ of $(\ox,0)$. We obtain from Lemma~\ref{Prop1} that 
\[
\la z,w\ra\ge \frac{\ell}{2}\|w\|^2\quad  \mbox{whenever} \quad z\in D\partial \varphi(u|v)(w), (u,v)\in \gph \partial \varphi\cap(U\times V).
\]
By Theorem~\ref{CharT}, $\ox$ is a tilt-stable solution of $\varphi$. \end{proof}

Second subderivative is also an important tool in  \cite{OM23} to obtain some new characterizations of tilt stability, but the function $\varphi$ in that paper is assumed to be twice epi-differentiable. In this case, Corollary~\ref{Tilt2} follows directly from Theorem~\ref{CharT} and equality \eqref{ZW}. Our characterization \eqref{SubT} is different from those in \cite{OM23} and does not need the function $\varphi$ to be twice epi-differentiable.  

\section{Geometric characterizations of Lipschitz stability for convex optimization problems}
\setcounter{equation}{0}

In this section we study tilt stability of optimization problem \eqref{CPS}
\begin{equation*}
    \min_{x\in \XX}\quad \varphi(x)\eqdef f(x)+g(x),
\end{equation*}
where $f,g:\XX\to \oR$ are proper l.s.c. convex functions and  $f$ is twice continuously differentiable in  $\XX$ with full domain. We need some additional conditions on the function $g$. 

\begin{Definition}[Quadratic growth condition for a set]\label{QGCR} Let $g:\XX\to \oR$ be a proper l.s.c. convex function with $(\ox,\oy)\in \gph \partial g$. The function $g$ is said to satisfy the quadratic growth condition at $\ox$ for $\oy$ with modulus $\kk>0$ if there exists $\ve>0$ such that 
\begin{equation}\label{QGC}
 g(x)-g(\ox)-\la \oy, x-\ox\ra\ge \frac{\kk}{2} \left[{\rm dist}\, (x; (\partial g)^{-1}(\oy))\right]^2\quad \mbox{for all}\quad x\in \B_\ve(\ox).    
\end{equation}

Let $\mathcal{M}$ be a subset of $\XX^*$ containing $\oy$. We say $g$ satisfies the quadratic growth condition around $(\ox,\oy)\in \gph \partial g$ for $\mathcal{M}$ if there exist $\eta,\kk>0$ such that for any $(x,y)\in \gph \partial g\cap \B_\eta(\ox,\oy)$ with $y\in \mathcal{M}$ the function  $g$ satisfies the quadratic growth condition at $x$ for $y$ with the same modulus $\kk.$
\end{Definition}

The quadratic growth condition for convex functions \cite{BS00} is known to be equivalent to the \emph{\L{}ojasiewicz inequality with exponent $\frac{1}{2}$} \cite{BNPS17} and the  \emph{metric subregularity of the subdifferential}   \cite{AG08, DMN14, ZT95}. Many important classes of convex functions satisfy these conditions; see, e.g.,  \cite{BNPS17,CDZ17,ZS17}. The quadratic growth condition for a set $\mathcal{M}$ introduced here seems new. It means that the quadratic growth condition holds at $x$ around $\ox$ for any $y\in \partial \varphi(x) \cap \mathcal{M}$ around $\oy$ with the same modulus, i.e., $\partial \varphi$ is metrically subregular at $x$ for $y\in \mathcal{M}$ in the sense of \eqref{MS0} below.  Of course, if the function $g$ satisfies the quadratic growth condition at $\ox$ for $\oy$, it also satisfies the quadratic growth condition around $(\ox,\oy)$ for the trivial set $\mathcal{M}=\{\oy\}$ by applying the Fenchel-Young's identity \eqref{FY}.  Finding the nontrivial set $\mathcal{M}\neq \{\oy\}$ is not simple. Since the modulus is uniform  when $y\in \mathcal{M}$, computing the modulus $\kk$ in \eqref{QGC} will help us determine the set $\mathcal{M}$. Next we discuss a few major functions that meet the quadratic growth condition  for  some nontrivial sets $\mathcal{M}$.

\begin{Example}[Convex piecewise linear-quadratic  functions, {\cite[Definition~10.20]{RW98}}]\label{CQP}{\rm The proper l.s.c. convex function $g:\R^n\to \oR$ is called {\em convex piecewise linear-quadratic} if $\dom g$ is a union of finitely many polyhedral sets, for  each of which the function $g(x)$ has a quadratic expression $\frac{1}{2}\la Ax,x\ra+\la b,x\ra+c$ for some scalar $c\in \R$, vector $b\in \R^n$, and positive semi-definite matrix $A$. This class of functions covers all {\em convex piecewise linear functions} including some important convex regularizers such as the $\ell_1$ (sparsity) norm, the $\ell_\infty$ norm, and the {\em anisotropic total variation} \cite{VDFPD17}. It also  contains the {\em elastic net regularizer} \cite{ZH05}, the $\ell_1+\ell_2$ {\em regularizer} \cite{HG11}, and the {\em discrete Blake-Zisserman regularizer} \cite{BS87}.  In this case, it is well-known \cite[Proposition~12.30]{RW98} that the subdifferential mapping $\partial g$ is {\em piecewise polyhedral}, i.e., the graph of $\partial g$ is a union of finitely many polyhedral sets. Obviously, the inverse subdifferential mapping $(\partial g)^{-1}$ is also piecewise polyhedral. By the landmark result of Robinson \cite{R81} (see also \cite[Proposition~3H.1 and Theorem~3H.3]{DR09}), the mapping $\partial g$ is {\em metrically subregular} at any $\ox\in \dom g$ for any $\oy\in \partial g(\ox)$ with the same modulus $\ell>0$ in the sense that there exists some $
\ve>0$ such that
\begin{equation}\label{MS0}
{\rm dist}\, (x;(\partial g)^{-1}(\oy))\le \ell {\rm dist}\,(\oy;\partial g(x))\quad \mbox{for all}\quad x\in \B_\ve(\ox)\cap\dom \partial g.  
\end{equation}
This condition is known to guarantee the quadratic growth condition \eqref{QGC} with  the modulus  $\kk$ smaller but arbitrarily close to $\ell^{-1}>0$; see, e.g., \cite{AG08, DMN14, ZT95}. Thus the function $g$ satisfies the  quadratic growth condition around $(\ox,\oy)\in \gph \partial g$ for the whole space $\mathcal{M}=\R^n$ or any neighborhood of $\bar y$. \hfill$\triangle$

 }   
\end{Example}

\begin{Example}[$\ell_1/\ell_2$ norm]\label{L12}{\rm Suppose that $\mathcal{J}$ is a partition of $\{1,2, \ldots, n\}$ with $m$ distinct groups. The $\ell_1/\ell_2$ norm for  this partition (a.k.a. the group Lasso regularizer) is defined by 
\begin{equation}\label{l12}
    g(x)=\|x\|_{1,2} \eqdef\sum_{J\in \mathcal{J}}\|x_J\|\quad \mbox{for}\quad x\in \R^n,
\end{equation}
where $\|x_J\|$ is the regular Euclidean norm at $x_J\in \R^{|J|}$ for any $J\in \mathcal{J}$. The subdifferential of this norm is metrically subregular at any $\ox$ for any $\oy \in \partial \|\ox\|_{1,2}$ by \cite[Proposition~9]{ZS17}. By \cite[Theorem~3.3]{AG08}, the $\ell_1/\ell_2$ norm  satisfies the quadratic growth condition. Here we advance this property by showing that it satisfies the quadratic growth condition around  $(\ox,\oy)\in \gph \partial g $ for a nontrivial set $\mathcal{M}$. In order to find the set $\mathcal{M}$, we  compute the modulus $\kk$ in \eqref{QGC} first. For any $J\in \mathcal{J}$,  note that 
\begin{eqnarray}\label{Sub12}
    \partial \|x_J\|=\left\{\begin{array}{ll} \B_J\; &\mbox{if}\; x_J=0_J\in \R^{|J|}\\
    \frac{x_J}{\|x_J\|} &\mbox{otherwise},
    \end{array}\right.
\end{eqnarray}
where $\B_J$ is the unit Euclidean ball in $\R^{|J|}$. Let us write $\oy=(\oy_J)_{J\in \mathcal{J}}$ with $\oy_J\in \partial \|\ox_J\|$. It follows from \eqref{Sub12} that 
\begin{eqnarray}\label{ISub}
(\partial \|\cdot\|)^{-1}(\oy_J)=\left\{\begin{array}{ll} \R_+\oy_J&\quad  \mbox{if}\quad \|\oy_J\|=1\\
\{0\} &\quad  \mbox{if} \quad \|\oy_J\|<1.\end{array}\right.
\end{eqnarray}
For any $x_J\in \R^{|J|}$, we have
\begin{equation}\label{QG1}
    \|x_J\|-\|\ox_J\|-\la \oy_J,x_J-\ox_J\ra=\|x_J\|-\la \oy_J,x_J\ra=\|x_J\|(1-\|\oy_J\|\cos\theta_J),
\end{equation}
where $\theta_J$ is the angle between $x_J$ and $\oy_J$. Let us consider two cases of \eqref{ISub}.

{\bf Case I.} If $\|\oy_J\|=1$, note from \eqref{ISub} that 
\[
{\rm dist }\,(x_J; (\partial \|\cdot\|)^{-1}(\oy_J))={\rm dist }\,(x_J;\R_+\oy_J).
\]
If $\theta_J\in [0,\frac{\pi}{2}]$, we have $
{\rm dist }\,(x_J;\R_+\oy_J)=\|x_J\|\sin \theta_J. $
This together with  \eqref{QG1} gives us that 
\begin{eqnarray}\label{QG2}\begin{array}{ll}
    \|x_J\|-\|\ox_J\|-\la \oy_J,x_J-\ox_J\ra&=\disp\|x_J\|(1-\cos\theta_J)=2\|x_J\|\sin^2(\theta_J/2)\\
    &=\disp\dfrac{1}{2\|x_J\|\cos^2(\theta_J/2)}\left[{\rm dist }\,(x_J; (\partial \|\cdot\|)^{-1}(\oy_J))\right]^2\\
    &\ge \disp\dfrac{1}{2\|x\|}\left[{\rm dist }\,(x_J; (\partial \|\cdot\|)^{-1}(\oy_J))\right]^2.
    \end{array}
\end{eqnarray}
If $\theta_J\in [\frac{\pi}{2},\pi]$,  note that ${\rm dist }\,(x_J; \R_+\oy_J)=\|x_J\|$. It follows from \eqref{QG1} that 
\begin{equation}\label{QG3}
\|x_J\|-\|\ox_J\|-\la \oy_J,x_J-\ox_J\ra=\|x_J\|(1-\cos\theta_J)\ge \|x_J\|\ge  \dfrac{1}{\|x\|}\left[{\rm dist }\,(x_J; (\partial \|\cdot\|)^{-1}(\oy_J))\right]^2.    
\end{equation}

{\bf Case II.} If $\|\oy_J\|<1$, we obtain from \eqref{ISub} that ${\rm dist }\,(x_J; (\partial \|\cdot\|)^{-1}(\oy_J))=\|x_J\|$. Combing this with \eqref{QG1} gives us that 
\begin{equation}\label{QG4}
\|x_J\|-\|\ox_J\|-\la \oy_J,x_J-\ox_J\ra\ge \|x_J\|(1-\|\oy_J\|)\ge \dfrac{1-\|\oy_J\|}{\|x\|}\left[{\rm dist }\,(x_J; (\partial \|\cdot\|)^{-1}(\oy_J))\right]^2.
\end{equation}
We are ready to show that the $\ell_1/\ell_2$ norm satisfies the quadratic growth condition. Define the following {\em subdominant number} $\gg$ of $\oy$:
\begin{equation}\label{gg}
    \gg=\gg(\oy)\eqdef\|\oy_\mathcal{H}\|_{\infty,2}=\max\{\|\oy_J\||\, J\in \mathcal{H}\}<1 
\end{equation}
if  $\mathcal{H}:=\{J\in\mathcal{J}|\; \|\oy_J\|<1\}\neq \emptyset$ and $\gg=0$ if $\mathcal{H}=\emptyset$. It follows from \eqref{QG2}, \eqref{QG3}, and \eqref{QG4} that 
\begin{eqnarray}\label{QG5}\begin{array}{ll}
\|x\|_{1,2}-\|\ox\|_{1,2}-\la \oy,x-\ox\ra
\disp\ge \frac{1-\gg}{2\|x\|}\left[{\rm dist }\,(x; (\partial \|\cdot\|_{1,2})^{-1}(\oy))\right]^2 
\end{array}.
\end{eqnarray}
 This shows the (global) quadratic growth condition at $\ox$ for $\oy$ with an explicit modulus. 

To find the set $\mathcal{M}$, we need to make sure that the subdominant numbers $\gg(y)$ should be close to $\gg$ and should not be close to $1$, when $y\in \mathcal{M}$ is around $\oy$.  A possible choice for $\mathcal{M}$ is 
\begin{equation}\label{M12}
\mathcal{M}_{1,2}\eqdef\{y\in \R^n|\; \|y_J\|=1, J\in \mathcal{K}\}\quad \mbox{with}\quad \mathcal{K}:=\mathcal{J}\setminus\mathcal{H}. 
\end{equation}
If $\mathcal{H}=\emptyset,$ we have  $\mathcal{K}=\mathcal{J}$. This is a trivial case as $\gg(y)=0$ for any $\in \mathcal{M}_{1,2}$. It suffices to consider the case $\mathcal{H}\neq\emptyset$.
For any $\ve\in (0,1-\gg)$ and $(x,y)\in \gph \partial \|\cdot\|_{1,2}\cap \B_\ve(\ox,\oy)$ with $y\in \mathcal{M}_{1,2}$, observe that $\oy$ and $y$ have the same index set $\mathcal{K}$.  Indeed, for any $J\in \mathcal{H}$, we have 
\[
\|y_J\|\le \|y_J-\bar y_J\|+\|\bar y_J\|\le \|y-\bar y\|+\gamma\le\ve+\gamma<1. 
\]
As $y\in \mathcal{M}_{1,2}$, the above inequality tells us that 
\[
\{J\in \mathcal{J}|\; \|y_J\|=1\}=\mathcal{K}.
\]
 It follows that 
\[
|\gg(y)-\gg(\oy)|=|\|y_\mathcal{H}\|_{\infty,2}-\|\oy_\mathcal{H}\|_{\infty,2}|\le \|y_\mathcal{H}-\oy_\mathcal{H}\|_{\infty,2}\le \ve. 
\]
Moreover, for any $u\in \B_\ve(x)\subset \B_{2\ve}(\ox)$, we derive from \eqref{QG5} that 
\[\begin{array}{ll}
\|u\|_{1,2}-\|x\|_{1,2}-\la y,u-x\ra\disp\ge \frac{1-\gg-\ve}{2\|\ox\|+4\ve}\left[{\rm dist }\,(u; (\partial \|\cdot\|_{1,2})^{-1}(y))\right]^2.
\end{array}
\]
As $\ve$ can be chosen arbitrarily small, the latter inequality shows that $\|\cdot\|_{1,2}$ satisfies the quadratic growth condition around $(\ox,\oy)$ for $\mathcal{M}_{1,2}$ with the same modulus. \hfill$\triangle$
}
\end{Example}

\begin{Example}[The indicator of the positive semidefinite cone $\mathbb{S}^n_+$]\label{SDP}{\rm  Let $\XX=\mathbb{S}^n$ be the space of all $n\times n$ symmetric matrices and $\mathbb{S}^n_+$ be the set of all positive semidefinite matrices. The indicator function $\delta_{\mathbb{S}^n_+}(X)$ is defined by $0$ if $X\in \mathbb{S}^n_+$ and $\infty$ otherwise. This function is especially important in  semidefinite programming. It is also known that $\delta_{\mathbb{S}^n_+}(\cdot)$ satisfies the quadratic growth condition \cite[Theorem~2.4]{C16}. For any $(\OX,\OY)\in \gph N_{\mathbb{S}^n_+}(\cdot)$, we have the spectral decomposition of $\OX$ and $\OY$:
\begin{equation}\label{SD}
\OX=\Bar P\begin{pmatrix}\Bar\Lm_\al&0&0\\0&0_\beta&0\\0&0&0_\eta\end{pmatrix}\Bar P^T\qquad\mbox{and}\qquad \OY=  \Bar P\begin{pmatrix}0_\al&0&0\\0&0_\beta&0\\0&0&\Bar \Lm_\eta\end{pmatrix}\Bar P^T\in N_{\mathbb{S}^n_+}(\OX),
\end{equation}
where $\{\al,\beta,\eta\}$ forms a partition of $\{1, \ldots,n\}$, 
$\Bar \Lm_\al={\rm diag}(\bar \lm_1, \ldots, \bar \lm_{|\al|})$ of positive non-increasing eigenvalues of $\OX$, $\Bar \Lm_\eta={\rm diag}(\bar \lm_{|\al|+|\beta|+1}, \ldots, \bar \lm_{n})$ of negative non-increasing eigenvalues of $\OY$, and $\Bar P$ is an orthogonal matrix. Let us consider the case $\OY\neq 0$ first.  \cite[Theorem~2.4]{C16} shows that
\begin{equation}\label{QGCIn}
\delta_{\mathbb{S}^n_+}(X)-\delta_{\mathbb{S}^n_+}(\OX)-\la \OY,X-\OX\ra\ge \frac{-\bar \lm_{|\al|+|\beta|+1}}{2\bar \lm_1+\frac{3}{2}+n}\big[{\rm dist}(X;  (N_{\mathbb{S}^n_+})^{-1}(\OY))\big]^2\quad \mbox{for all}\quad X\in \B_\ve(\OX) 
\end{equation}
for some $\ve>0$. The eigenvalue $\bar \lm_{|\al|+|\beta|+1}$ is the {\em subdominant eigenvalue} of $\OY$.  Let us define 
\begin{equation}\label{MS}
\mathcal{M}_\mathbb{S}\eqdef\{Y\in \mathbb{S}^n|\; \lm_i(Y)=0, i\in \al\cup\beta\}, 
\end{equation}
where $\lm_1(Y),\ldots,\lm_n(Y)$ are the eigenvalue functions in non-increasing order of $Y\in \mathbb{S}^n$. For any $(X,Y)\in \gph N_{\mathbb{S}^n_+}(\cdot)\cap\B_\ve(\OX,\OY)$ and  for some small $\ve>0$ with $Y\in \mathcal{M}_{\mathbb{S}}$, it follows from \eqref{QGCIn} that there exists some $\nu>0$ (possibly depending on $(X,Y)$) such that 
\begin{equation*}
\delta_{\mathbb{S}^n_+}(U)-\delta_{\mathbb{S}^n_+}(X)-\la Y,U-X\ra\ge \frac{- \lm_{|\al|+|\beta|+1}(Y)}{ 2\lm_1(X)+\frac{3}{2}+n}\big[{\rm dist}(X;  (N_{\mathbb{S}^n_+})^{-1}(Y))\big]^2\quad \mbox{for}\quad U\in \B_\nu(X).  
\end{equation*}
As $\lm_1(\cdot)$ and $\lm_{|\al|+|\beta|+1}(\cdot)$ are Lipschitz continuous functions with modulus $1$, we derive from the latter that 
\begin{equation*}
\delta_{\mathbb{S}^n_+}(U)-\delta_{\mathbb{S}^n_+}(X)-\la Y,U-X\ra\ge \frac{- \bar \lm_{|\al|+|\beta|+1}-\ve}{ 2\bar \lm_1+2\ve+\frac{3}{2}+n}\big[{\rm dist}(X;  (N_{\mathbb{S}^n_+})^{-1}(Y))\big]^2\quad \mbox{for}\quad U\in \B_\nu(X).  
\end{equation*}
By choosing $\ve$ small enough, the above inequality tells us that the indicator function $\delta_{\mathbb{S}_+}(\cdot)$ satisfies the quadratic growth condition for the set $\mathcal{M}_\mathbb{S}$. When $\OY=0$, we may choose $\mathcal{M}_\mathbb{S}=\{0\}$. In this case, the quadratic growth condition around $(\OX,\OY)$ for $\mathcal{M}_\mathbb{S}$ is also true with any positive modulus. \hfill$\triangle$
 }   
\end{Example}

 It is worth noting that the choice of $\mathcal{M}$ in Definition~\ref{QGCR} is not necessarily unique; for instance, $\mathcal{M}$ can be chosen as any subset of $\R^n$ containing $\oy$ in Example~\ref{CQP} or simply $\{\oy\}$ in Example~\ref{L12}. The nuclear norm also satisfies the quadratic growth condition for some nontrivial set; see our Lemma~\ref{QG} for further details and discussions.  Next, we recall the following lemma taken from \cite[Lemma~3.2]{FNP23} that provides an estimate for the second subderivative of $g$.

\begin{Lemma}[Quadratic growth condition and second subderivative]\label{QGSS} Suppose that the function $g$ satisfies the quadratic growth condition at $\ox$ for $\oy\in \partial g(\ox)$ with modulus $\kk>0$. Then we have 
\begin{equation}\label{QGD1}
    d^2 g(\ox|\oy)(w)\ge \kk\left[{\rm dist}\, (w;T_{(\partial g)^{-1}(\oy)}(\ox))\right]^2\qquad \mbox{for all}\qquad w\in \XX.
\end{equation}
Consequently, 
\begin{equation}\label{Ker0}
    \Ker d^2 g(\ox|\oy)=\{w\in \XX|\;d^2 g(\ox|\oy)(w)=0\}= T_{(\partial g)^{-1}(\oy)}(\ox).
\end{equation}

\end{Lemma}


To use the quadratic growth condition for a set $\mathcal{M}$ in Definition~\ref{QGCR}, we need $(x,y)\in\gph \partial g$ and $y\in \mathcal{M}$. If $y \notin \mathcal{M}$, we will try to ``push'' it onto $\mathcal{M}$ by approximating it by some element in $\mathcal{M}\cap \partial g(x)$. Such an approximation is introduced below.

\begin{Definition}[Relative approximations of subgradients by a set]\label{RelA} Let $g:\XX\to \oR$ be a proper l.s.c. convex function with $(\ox,\oy)\in \gph \partial g$ and $\mathcal{M}$ be a subset in $\XX^*$ containing $\oy$. We say the subdifferential mapping $\partial g$ has relative approximations  by $\mathcal{M}$ around $(\ox,\oy)$ if there exists $\ve\in (0,1)$ such that for any $(x,y)\in \gph \partial g\cap \B_\ve(\ox,\oy)$, we can find $\lm\in (1-\ve,1]$ and $\hat y\in \partial g(x)\cap\mathcal{M}$ and $\tilde y\in \partial g(x)$ satisfying 
\begin{equation}\label{CA}
   \lm \hat y+(1-\lm)\tilde y=y.
\end{equation} 
Here, $\hat y$ is called a relative approximation of $y$  by $\mathcal{M}$.
\end{Definition}

In later results of this section, we need  the function $g$ to satisfy both the quadratic growth condition for a set $\mathcal{M}$ and the relative approximations of subgradients  by the same set. When $\mathcal{M}$ is the whole space $\XX^*$, the above definition is trivial. This is the case of convex piecewise-linear quadratic functions in Example~\ref{CQP}.  The choice of $\mathcal{M}$ in that example is also not unique, as we can choose $\mathcal{M}$ to be any neighborhood of $\oy$ and \eqref{CA} holds with $\lm=1$, $\hat y=\tilde y=y$.   In general, the set $\mathcal{M}$ in Definition~\ref{RelA} should have a special property that $\partial g(x)\cap \mathcal{M}\neq \emptyset$ for $x$ near to the given point $\ox$. This does not allow us to choose $\mathcal{M}$ as the trivial set $\{\oy\}$ like the case of Definition~\ref{QGC}. 

\begin{Example}[Relative approximations of the subgradients of $\ell_1/\ell_2$ norm]\label{L21}{\rm Let $\|\cdot\|_{1,2}$ be the $\ell_1/\ell_2$ norm defined in \eqref{l12}. Given a point $(\ox,\oy)\in \gph \partial\|\cdot\|_{1,2}$, we claim that the subdifferential mapping  $\partial\|\cdot\|_{1,2}$ always has relative approximations  by the set $\mathcal{M}_{1,2}$ defined in \eqref{M12}. Pick any $\ve\in (0, 1-\gg)$ and any $(x,y)\in \gph \partial \|\cdot\|_{1,2}\cap\B_\ve(\ox,\oy)$ with $\gg$ in \eqref{gg}. If $y\in \mathcal{M}_{1,2}$, we can choose $\lm=1$, $\hat y=\tilde y=y$ so that \eqref{CA} holds. If $y\not \in \mathcal{M}_{1,2}$, define 
\[
\mathcal{K}_1=\{J\in\mathcal{J}|\; \|y_J\|=1\}. 
\]
As $y\in \partial \|x\|_{1,2}$ and $\|y-\oy\|\le \ve$,  observe that $\mathcal{K}_1\subset \mathcal{K}$ when $\ve$ is sufficiently small. Since $y\notin \mathcal{M}_{1,2}$,  we have $\mathcal{I}_1\eqdef\mathcal{K}\setminus \mathcal{K}_1\not=\emptyset$ and thus $\lm\eqdef\min_{J\in \mathcal{I}_1}\|y_J\|<1.$  
For any $J\in \mathcal{I}_1$, note further that  
\[
\ve\ge\|y_J-\oy_J\|\ge \|\oy_J\|-\|y_J\|\ge 1-\|y_J\|,
\]
which implies that $\|y_J\|\ge 1-\ve$ for any $J\in \mathcal{I}_1$, i.e., $\lm\ge 1-\ve$ and $y_J\neq 0$.  

It is time to find $\hat y$ and $\tilde y$ satisfying \eqref{CA}. Define
\begin{equation*}
    \hat y=\left\{\begin{array}{ll}
    \hat y_J= y_J \quad &\mbox{if}\quad J\in \mathcal{J}\setminus \mathcal{I}_1 \\
    \hat y_J= \frac{y_J}{\|y_J\|}\quad  &\mbox{if}\quad J\in \mathcal{I}_1
    \end{array}\right.\qquad \mbox{and}\qquad \tilde y=\left\{\begin{array}{ll}
    \tilde y_J= y_J \quad &\mbox{if}\quad J\in \mathcal{J}\setminus \mathcal{I}_1 \\
    \tilde y_J= \frac{\|y_J\|-\lm}{1-\lm}\cdot \frac{y_J}{\|y_J\|}\quad  &\mbox{if}\quad J\in \mathcal{I}_1
    \end{array}\right.
\end{equation*}
As $\|y_J\|<1$ for $J\in \mathcal{I}_1$, we have $x_J=0$ for $J\in \mathcal{I}_1$. It follows that $\hat y,\tilde y\in \partial \|x\|_{1,2}$. Note further that $\hat y\in \mathcal{M}_{1,2}$ and that $\lm \hat y+(1-\lm)\tilde y=y$. This verifies that $\partial \|\cdot\|_{1,2}$ has relative approximations  by $\mathcal{M}$ around $(\ox,\oy)$. \hfill$\triangle$
}
\end{Example}

\begin{Example}[Relative approximations of  $N_{\mathbb{S}^n_+}(\cdot)$]{\rm  We claim that the normal cone mapping $N_{\mathbb{S}^n_+}(\cdot)$ has  relative approximations  by the set $\mathcal{M}_\mathbb{S}$ in \eqref{MS}. Given a pair $(\OX,\OY)\in \gph N_{\mathbb{S}^n_+}$ with decomposition~\eqref{SD}, for any $(X,Y)\in \gph N_{\mathbb{S}^n_+}\cap\B_\ve(\OX,\OY)$ with  $\ve>0$, it follows from \eqref{SD} that 
\begin{equation}\label{SD2}
X=P\begin{pmatrix}\Lm_{\al^\prime}&0&0\\0&0_{\beta^\prime}&0\\0&0&0_{\eta^\prime}\end{pmatrix} P^T\qquad\mbox{and}\qquad Y=  P\begin{pmatrix}0_{\al^\prime}&0&0\\0&0_{\beta^\prime}&0\\0&0& \Lm_{\eta^\prime}\end{pmatrix} P^T 
\end{equation}
for some orthogonal matrix $P$. 
For sufficiently small $\ve>0$, $\eta\subset\eta^\prime$. If $\eta=\eta^\prime$, i.e.,  $Y\in \mathcal{M}_\mathbb{S}$, we choose $\Hat Y=Y$ as a relative approximation of $Y$  by $\mathcal{M}_\mathbb{S}$. If $\eta\neq\eta^\prime$, we define 
\begin{equation}
\Hat Y=  P\begin{pmatrix}0_{\al}&0&0\\0&0_{\beta}&0\\0&0& \Lm_{\eta}\end{pmatrix} P^T\quad \mbox{and}\quad \Tilde Y= P\begin{pmatrix}0_{\al^\prime\cup\beta^\prime}&0&0\\0&\sigma^{-1}\Lm_{\eta^\prime\setminus\eta}&0\\0&0& \Lm_{\eta}\end{pmatrix} P^T
\end{equation}
with $\sigma=|\lm_{|\al|+|\beta|}(Y)|> 0$. As $\bar\lm_{|\al|+|\beta|}(\OY)=0$, we have
\[
0< \sigma=|\lm_{|\al|+|\beta|}(Y)-\bar\lm_{|\al|+|\beta|}(\OY)|\le \|Y-\OY\|\le \ve. 
\]
Observe further that $\Hat Y\in N_{\mathbb{S}^n_+}(X)\cap \mathcal{M}_\mathbb{S}$, $\Tilde Y\in N_{\mathbb{S}^n_+}(X)$, and that 
$Y=(1-\sigma)\Hat Y+\sigma \Tilde Y$. This shows that $\Hat Y$ is a relative approximation of $Y$  by $\mathcal{M}_\mathbb{S}$. \hfill$\triangle$
}
 \end{Example}

We are ready to set up the central result of this section, which establishes a geometric characterization of tilt stability for problem \eqref{CPS}. 

\begin{Theorem}[Geometric characterization of tilt stability]\label{GeoTilt} Let $\ox$ be an optimal solution of problem \eqref{CPS}. Suppose that there exists a set $\mathcal{M}\subset\XX^*$ containing $\oy=-\nabla f(\ox)\in \partial g(\ox)$ such that the function $g$ satisfies the quadratic growth condition around $(\ox,\oy)$ for  $\mathcal{M}$ and that the subdifferential mapping $\partial g$ has relative approximations  by $\mathcal{M}$ around $(\ox,\oy)$. Then $\ox$ is a tilt-stable solution if and only if 
\begin{equation}\label{Geo}
\Ker \nabla^2 f(\ox)\cap \mathfrak{T}_\mathcal{M}=\{0\},
\end{equation}
where the set $\mathfrak{T}_\mathcal{M}$ is defined by
\begin{equation}\label{T}
\mathfrak{T}_\mathcal{M}\eqdef\Limsup_{(x,y)\stackrel[y\in \mathcal{M}]{{\rm gph}\,\partial g}{\longrightarrow}(\ox,\oy)} T_{(\partial g)^{-1}(y)}(x),
\end{equation}
where the {\rm Limsup} is in the Painlev\'e-Kuratowski sense, i.e., 
\[
\mathfrak{T}_\mathcal{M}=\left\{w\in \XX|\; \exists\, (x_k,y_k)\in \gph \partial g\; {\rm and }\; w_k\in T_{(\partial g)^{-1}(y_k)}(x_k): y_k\in \mathcal{M}, (x_k,y_k,w_k)\to (\ox,\oy,w)\right\}.
\]
\end{Theorem}
\begin{proof} Let us start by supposing that $\ox$ is a tilt-stable solution of problem \eqref{CPS}. For any $w$ in the set $\Ker \nabla^2 f(\ox)\cap \mathfrak{T}_\mathcal{M}$, we claim that $w=0$ to verify condition \eqref{Geo}. Indeed, take any sequences $w_k\to w$ and $(x_k,y_k)\stackrel{{\rm gph}\,\partial g}{\longrightarrow}(\ox,\oy)$ with $y_k\in \mathcal{M}$ and $w_k\in T_{(\partial g)^{-1}(y_k)}(x_k)$. Define $v_k:=\nabla f(x_k)+y_k\in \partial \varphi(x_k)$ and note that $v_k\to 0$. By Corollary~\ref{Tilt2}, we find some $\ell>0$ such that
\[
\ell \|w_k\|^2\le d^2\varphi(x_k|\,v_k)(w_k)=\la \nabla^2 f(x_k) w_k,w_k\ra+d^2 g(x_k|\, y_k)(w_k) 
\]
for sufficiently large $k$. As $w_k\in T_{(\partial g)^{-1}(y_k)}(x_k)$, it follows from  Lemma~\ref{QGSS} that $d^2 g(x_k| y_k)(w_k)=0$. The above inequality gives us that 
\[
\ell \|w_k\|^2\le \la \nabla^2 f(x_k) w_k,w_k\ra.
\]
By passing to the limit, we have $\ell\|w\|^2\le \la \nabla^2 f(\ox) w,w\ra=0$, i.e., $w=0$. This ensures \eqref{Geo}. 

To justify the converse implication, suppose that condition \eqref{Geo} holds. If $\ox$ is not a tilt-stable solution of problem \eqref{CPS}, by Corollary~\ref{Tilt2} again there exist a sequence $(x_k,y_k)\stackrel{{\rm gph}\,\partial g}{\longrightarrow}(\ox,\oy)$ and $w_k\in \XX$ with $\|w_k\|=1$ such that 
\begin{equation}\label{Ctr}
    \frac{1}{k}\ge d^2\varphi(x_k|\; \nabla f(x_k)+y_k)(w_k)= \la\nabla^2 f(x_k) w_k,w_k\ra+d^2 g(x_k|\, y_k)(w_k).
\end{equation}
Since the function $g$ satisfies the quadratic growth condition  around $(\ox,\oy)$ for  $\mathcal{M}$, there exist $\ve, \kk>0$ such that for any $(x,y)\in \gph \partial g\cap \B_{2\ve}(\ox,\oy)$ and $y\in \mathcal{M}$, the function $g$ satisfies the quadratic growth condition at any $x$ for  $y$ with the same modulus $\kk>0$. Moreover, as $\partial g$ has relative approximations  by $\mathcal{M}$ around $(\ox,\oy)$, we find $\lm_k\in (1-\ve,1]$, $\hat y_k\in \partial g(x_k)\cap \mathcal{M}$, and $\tilde y_k\in \partial g(x_k)$ such that $\lm_k \hat y_k+(1-\lm_k)\tilde y_k=y_k$ when $k$ is sufficiently large. Note that  $d^2 g(x_k| \cdot)(w_k)$ is a concave function. It follows from \eqref{Ctr} that 
\[\begin{array}{ll}
\frac{1}{k}&\disp\ge \la\nabla^2 f(x_k) w_k,w_k\ra+\lm_k d^2 g(x_k|\, \hat y_k)(w_k)+(1-\lm_k)d^2 g(x_k|\, \tilde y_k)(w_k).\\
&\disp\ge \la\nabla^2 f(x_k) w_k,w_k\ra+(1-\ve) d^2 g(x_k|\, \hat y_k)(w_k).
\end{array}
\]
This together with Lemma~\ref{QGSS} implies that 
\[
\frac{1}{k}\ge \la\nabla^2 f(x_k) w_k,w_k\ra+(1-\ve) \kk \left[{\rm dist}\, (w_k;T_{(\partial g)^{-1}(\hat y_k)}(x_k))\right]^2.
\]
Hence, there exists $u_k\in T_{(\partial g)^{-1}(\hat y_k)}(x_k)$ such that 
\[\frac{1}{k}\ge \la\nabla^2 f(x_k) w_k,w_k\ra+(1-\ve) \kk \|w_k-u_k\|^2\ge (1-\ve) \kk \|w_k-u_k\|^2. 
\]
By passing to subsequences, we may suppose that both $w_k$  and $u_k$ converge to some $w\in \XX$ with $\|w\|=1$. The later inequalities also imply that $0\ge \la\nabla^2 f(\ox) w,w\ra$, i.e., $w\in \Ker \nabla^2 f(\ox)$. As $\hat y_k\in \partial g(x_k)\cap \mathcal{M}$ and $u_k\in T_{(\partial g)^{-1}(\hat y_k)}(x_k)$, we have $w\in  \mathfrak{T}_\mathcal{M}$. It follows that $w\in \Ker \nabla^2 f(\ox)\cap \mathfrak{T}=\{0\}$, which is a contradiction. Hence, $\ox$ must be a tilt-stable solution of problem \eqref{CPS}. 
\end{proof}

As the contingent cone $T_{(\partial g)^{-1}(y)}(x)$ is usually easy to compute, the set $\mathfrak{T}_\mathcal{M}$ looks more computable than the generalized Hessian used in \cite{PR98} and  some other second-order structures in \cite{CHN18,DMN14,MN15} to characterize tilt stability. These second-order structures often lead to the computation of the ``sigma-term" or the ``curvature'' on regularizers $g$ that are also complicated to calculate.  The ``Limsup'' in \eqref{T} relies on the set $\mathcal{M}$, but at this stage, we do not know if $\mathfrak{T}_\mathcal{M}$  actually depends on the choice of $\mathcal{M}$.

\begin{Remark}[Subspace containing derivative]{\rm 
In \cite{GO22}, Gfrerer and Outrata introduced a new characterization of tilt stability via the so-called {\em subspace containing derivative} (SCD, in brief) \cite[Definition~3.3 and Theorem~7.9]{GO22} for the subdifferential mapping $\partial \varphi$, which has strong connection with the {\em quadratic bundle} of $\varphi$ introduced by Rockafellar \cite{R23}; see aslo \cite[Lemma 3.31 and Proposition~3.33]{GO22} and the discussion therein. The SCD of $\partial \varphi$ takes the ``Limsup'' of the graphical derivative $D\partial \varphi$ over the set $\mathcal{O}_{\partial\varphi}$ of $(x,y)$ at which the graph of  $D\partial \varphi(x|\,y)$ (or the tangent cone $T_{{\rm gph}\, \partial\varphi}(x,y)$) is a subspace. The ``Limsup'' in \eqref{T} shares some similar idea, as we also require $y$ to belong to a special set $\mathcal{M}$. However, our set $\mathcal{M}$ is mostly about ``$y$'' and  relates more about the quadratic growth condition. Moreover, the tangent cone in \eqref{T} is taken on the subdifferential set $(\partial g)^{-1}(y)= \partial g^{*}(y)$, not on the whole graph of $\partial\varphi$.  Although it turn outs that $\mathfrak{T}_\mathcal{M}$ is a subspace in different classes of regularizers $g$ in this paper, we do not know  if  the $T_{{\rm gph}\, \partial\varphi}(x,y)$ is a subspace or not when $y\in \mathcal{M}\cap \partial g(x)$. SCDs on various norm functions (not including the nuclear norm) are computed in \cite{G24}. The characterization of tilt stability for the group Lasso problem \eqref{Las12} in Corollary~\ref{LasT} below can be obtained from \cite[Proposition~3.15 and Lemma~3.16]{G24} together with the possible extension of \cite[Example~3.11]{G24} to the $\ell_1/\ell_2$ norm. But the characterization of Lipschitz stability for problem \eqref{Las12} in Theorem~\ref{Lip12} is new. A deeper analysis understanding the connection between these two approaches and particularly between the set $\mathcal{O}_{\partial \varphi}$ and $\mathcal{M}$ will be helpful for further research in this direction.

As the function $g$ satisfies the quadratic growth condition around $(\ox,\oy)$ for $\mathcal{M}$ as in Theorem~\ref{GeoTilt}, it follows from \cite[Corollary 4.2]{GO16} that 
\begin{equation}\label{eq:TgD}
T_{\partial g^*(y)}(x)=D\partial g^*(y,x)(0)
\end{equation}
for any $(x,y)\in \gph \partial g\cap \mathcal{M}$ around $(\ox,\oy)$. The limsup in \eqref{T} tells us that the set $\mathfrak{T}_\mathcal{M}$ is a subset of $D^\sharp \partial g^*(\oy, \ox)(0)$, which is the so-called {\em outer limiting graphical derivative} of $\partial g^*$ at $(\oy, \ox)$ \cite[Definition~3.8]{GO22}:
\begin{eqnarray}\label{def:OL}
 D^\sharp \partial g^*(\oy, \ox)(0)=\left\{w\in \XX|\; \exists\, (w_k,v_k)\st{\XX\times \XX^*}\to (w,0), (x_k,y_k)\st{{\rm gph}\,\partial g}\to (\ox,\oy): w_k\in D\partial g^*(y_k,x_k)(v_k)\right\}.
\end{eqnarray}
For the particular class of convex piecewise linear-quadratic functions, it is not hard to show that  $\mathfrak{T}_\mathcal{M}$ and $D^\sharp \partial g^*(\oy, \ox)(0)$ are the same, due to the piecewise polyhedral property of its subdifferential and the choice $\mathcal{M}$ as the whole space as discussed in Example~\ref{CQP}. ``Whether are they equal for other classes of $g$?'' is the question that we do not have the complete answer yet, as the computation of \eqref{def:OL} is rather involved in our opinion. There are two main differences in their definitions:  the choice of sequence $v_k$ in \eqref{def:OL} is fixed as $0$ in \eqref{T} and $y_k$ in \eqref{def:OL} lies on a special set $\mathcal{M}$ in \eqref{T}. They essentially make the set  $\mathfrak{T}_\mathcal{M}$ more computable in several important cases in this paper; for example, when $g$ is the nuclear norm, it is still hard to obtain the computation of $D\partial g^*(y_k,x_k)(v_k)$ before taking the limit as in \eqref{def:OL}, but the formula $D\partial g^*(y_k,x_k)(0)$ is quite simple (Lemma~\ref{Tang}) and allows us to fully calculate the formula $\mathfrak{T}_\mathcal{M}$ in \eqref{LimTang}; see our Section~4 for further details. Moreover, if the computation of $D\partial g^*(y_k,x_k)(v_k)$ is available, we may use the characterization in Theorem~2.3(iii) from \cite{CHN18} to derive a full characterization of tilt stability. 
}
  \hfill$\triangle$    
\end{Remark}

\begin{Corollary}[Tilt stability for group Lasso problem]\label{LasT} Consider the following group Lasso problem
\begin{equation}\label{Las12}
    \min_{x\in \R^n}\quad f(x)+\|x\|_{1,2},
\end{equation}
where $f:\R^n\to \R$ is a convex function that is continuously twice differentiable and $\|\cdot\|_{1,2}$ is the $\ell_1/\ell_2$ norm defined in \eqref{l12}. Then $\ox$ is a tilt-stable solution if and only if $\oy=-\nabla f(\ox)\in \partial \|\ox\|_{1,2}$ and 
\begin{equation}\label{Tilt12}
   \Ker \nabla^2 f(\ox)\cap \mathcal{V}=\{0\},
\end{equation}
where $\mathcal{V}$ is the subspace of $\R^n$ defined by
\begin{equation}\label{V}
    \mathcal{V}\eqdef\{w\in \R^n|\; w_J\in \R\{\oy_J\}\;\; \mbox{if}\;\; J\in \mathcal{K}\;\; \mbox{and}\;\; w_J=0\;\; if\;\; J\in \mathcal{H}\}
\end{equation}
with $\mathcal{K}=\{J\in \mathcal{J}|\; \|\oy_J\|=1\}$ and  $\mathcal{H}=\{J\in \mathcal{J}|\; \|\oy_J\|_J<1\}$ . 
\end{Corollary}
\begin{proof} By Example~\ref{L12}, Example~\ref{L21}, and  Theorem~\ref{GeoTilt}, we only need to show that the set $\mathfrak{T}$ defined in  \eqref{T} is $\mathcal{V}$. Indeed, for any $(x,y)\in \gph \partial \|\cdot\|_{1,2}$ with $y\in \mathcal{M}_{1,2}$, we obtain from \eqref{ISub} that 
\begin{equation}\label{Tan}\begin{array}{ll}
T_{(\partial \|\cdot\|_{1,2})^{-1}(y_J)}(x_J)&=\left\{\begin{array}{ll}\left\{\begin{array}{ll}\R_+\{y_J\}\;\;&\mbox{if}\;\; x_J=0\\
\R \{y_J\}\;\;&\mbox{if}\;\; x_J\neq 0. 
\end{array}\right.\subset \R \{y_J\}\quad \mbox{for any}\quad J\in \mathcal{K}\\
T_{0_J}(x_J)=0_J\quad \mbox{for any}\quad J\in \mathcal{H}.\end{array}\right.
\end{array}
\end{equation}
It is easy to derive from \eqref{T} that $\mathfrak{T}_{\mathcal{M}_{1,2}}\subset \mathcal{V}$. To verify the opposite inclusion,  define  $\mathcal{I}\eqdef\{J\in \mathcal{J}|\; \ox_J\neq 0\}\subset \mathcal{K}$ and 
\begin{equation}\label{xy}
x^k=\left\{\begin{array}{ll}(x^k)_J=\ox_J\;\; &\mbox{if}\;\; J\in  \mathcal{I}\cup \mathcal{H}\\
(x^k)_J=\frac{1}{k}\oy_J\;\;&\mbox{if}\;\; J\in \mathcal{K}\setminus\mathcal{I}. 
\end{array}\right.\quad \quad \mbox{and}\quad y^k=\oy. 
\end{equation}
Note that $(x^k,y^k)\to (\ox,\oy)$, $y^k\in \partial \|x^k\|_{1,2}\cap \mathcal{M}_{1,2}$, and from \eqref{Tan} that  
\begin{equation}\label{V12}
T_{(\partial \|\cdot\|_{1,2})^{-1}(y^k)}(x^k)=T_{(\partial \|\cdot\|_{1,2})^{-1}(\oy)}(x^k)=\mathcal{V},
\end{equation}
which implies that $\mathcal{V}\subset \mathfrak{T}_{\mathcal{M}_{1,2}}$ and thus $\mathfrak{T}_{\mathcal{M}_{1,2}}=\mathcal{V}$. By Theorem~\ref{GeoTilt}, $\ox$ is a tilt-stable solution of problem \eqref{Las12} if and only if condition \eqref{Tilt12} is satisfied.
\end{proof}

When the $\ell_1/\ell_2$ norm reduces to the $\ell_1$ norm for Lasso problem, the above result can be obtained easily from \cite[Proposition~3.15, Lemma~ 3.16, and Example~3.8]{G24} with a different approach via the SC derivatives. For the case of $\ell_1/\ell_2$ norm as in Corollary~\ref{LasT}, the characterization of tilt stability in \eqref{Tilt12} can be also derived from \cite[Proposition~3.15 and Lemma~ 3.16]{G24} by computing the SC derivative for the $\ell_1/\ell_2$ norm, which is just the separable sum of  $\ell_2$ norm functions \cite[Example~3.11]{G24}. The choice of $(x_k,y_k)$ in the above computation is to make  $T_{(\partial \|\cdot\|_{1,2})^{-1}(y^k)}(x^k)$ to be a subspace; see also the proof of our Theorem~\ref{LipNu} for the case of nuclear norm. This is close to the idea of using the  {\em subspace containing derivative} in \cite{GO22}. The difference is that we do not compute the tangent cone to the whole graph of $\partial g$ as in \cite{GO22}, but only ``$y$'' part of it. 

As discussed in the Introduction, a particular problem of \eqref{CPS} is  
\begin{equation}\label{QP}
  \mathcal{P}(b,\mu)\qquad \min_{x\in \XX}\quad   \frac{1}{2\mu}\|\Phi x-b\|^2+g(x), 
\end{equation}
where $\Phi$ is a linear operator between two Euclidean spaces  $\XX$ and $\YY$ and $(b,\mu)\in \YY\times \R_{++}$ is a parameter pair around some fixed  $(\bar b,\bar \mu)\in \YY\times \R_{++}$. 
The solution set of this problem is defined in \eqref{Sp} and \eqref{GE} by
\begin{eqnarray}\label{S}
    S(b,\mu)={\rm argmin}\left\{\frac{1}{2\mu}\|\Phi x-b\|^2+g(x)|\; x\in \XX\right\}=\left\{x\in \XX|\; 0\in \frac{1}{\mu}\Phi^*(\Phi x-b)+\partial g(x)\right\}.
\end{eqnarray}
Next we will provide some analysis for Lipschitz continuity of this solution mapping. It is similar to the discussion after Definition~\ref{TS} that in our convex settings, the mapping $S$ is single-valued and Lipschitz continuous around $(\bar b,\bar \mu)$ for $\ox=S(\bar b,\bar \mu)$ if and only if it has a Lipschitz continuous single-valued localization around $(\bar b,\bar \mu)$ for $\ox=S(\bar b,\bar \mu)$.  Moreover, since the function $h(x, b,\mu)=\frac{1}{\mu}\Phi^*(\Phi x-b)$ is strictly differentiable at $(\ox, (\bar b,\bar \mu))$, $S$ has single-valued and Lipschitz continuous around $(\bar b,\bar \mu)$ for $\ox$ if the following mapping
\[
X^*\ni v\mapsto G(v)\eqdef\{x\in \XX|\;v\in h(\bar x, \bar b,\bar \mu)+ \nabla_xh(\ox, \bar b,\bar \mu)(x-\ox)+\partial g(x)\}
\]
single-valued and Lipschitz continuous around $0$ for $\ox=G(0)$; see, e.g., \cite[Corrollary~2B.9]{DR09}.  As $h$ is a linear function with respect to $x$, we have 
\[
h(\ox, \bar b,\bar \mu)+ \nabla_xh(\ox, \bar b,\bar \mu)(x-\ox)=h(x, \bar b,\bar \mu)=\frac{1}{\bar \mu}\Phi^*(\Phi x-\bar b).
\]
It follows that  
\[
G(v)={\rm argmin}\left\{\frac{1}{2\bar \mu}\|\Phi x-\bar b\|^2+g(x)-\la v,x\ra|\; x\in \XX\right\}. 
\]
Thus $G$ is single-valued and Lipschitz continuous around $0$ for $\ox=G(0)$ if and only if $\ox$ is a tilt-stable solution of problem $P(\bar b,\bar \mu)$. 
These analysis together with Theorem~\ref{GeoTilt} gives us a sufficient condition for local Lipschitz continuity of the solution mapping $S$ as below.

\begin{Corollary}[Lipschitz stability of solution mapping]\label{LipS12} Suppose that $\ox$ is an optimal solution of problem $\mathcal{P}(\bar b,\bar \mu)$ in  \eqref{QP}. Suppose further  that there exists a set $\mathcal{M}\subset \XX^*$ containing $\oy=-\frac{1}{\bar \mu}\Phi^*(\Phi \ox-\bar b)$ such that the function $g$ satisfies the quadratic growth condition around $(\ox,\oy)$ for  $\mathcal{M}$ and that the subdifferential mapping $\partial g$ has relative approximations  by $\mathcal{M}$ around $(\ox,\oy)$. Then  the solution mapping $S$ is single-valued and Lipschitz continuous around $(\bar b,\bar \mu)$ with $\ox=S(\bar b,\bar \mu)$ provided that 
\begin{equation}\label{Sb}
\Ker \Phi \cap \mathfrak{T}_\mathcal{M}=\{0\}. 
\end{equation}
\end{Corollary}

It is worth noting that \eqref{Sb} and tilt stability of problem~$P(\bar b,\bar \mu)$ are only sufficient  for the Lipschitz stability of  $S$ in general. For some special classes of $g$ \cite{DR09,CD19,MNR15,S06},  sufficient conditions for Lipschitz stability of $S$ can be obtained via the {\em Robinson's strong regularity} introduced in \cite{R80}. But again, these two kinds of stability  are  different; see, e.g., \cite{DR09} for further details and discussions. Moreover, Robinson's strong regularity is usually characterized via the so-called {\em Strong Second-Order Sufficient Condition} that may contain some second-order structures that are also involved to compute and verify, e.g., when $g$ is the nuclear norm.  Our approach is different via tilt stability. It does require some special properties on $g$, but the most important thing is condition \eqref{Sb} can be fully calculated. 

Next we will show that condition \eqref{Sb} is also necessary for the Lipschitz stability of solution mapping $S$ in \eqref{S} when the function $g$ is the  $\ell_1/\ell_2$ norm. Indeed, both of them are equivalent to the single-valuedness of solution mapping $S$. This is quite a surprise, as Lipschitz continuity is  not a free property for single-valued mappings in general. 

\begin{Theorem}[Full characterization of Lipschitz stability for group Lasso problem]\label{Lip12} Consider the following parametric group Lasso problem 
\begin{equation}\label{Las21}
    \min_{x\in\R^n}\quad \frac{1}{2\mu} \|\Phi x-b\|^2+\|x\|_{1,2},
\end{equation}
 where $\Phi:\R^n\to\R^m$ is a linear operator and $(b,\mu)\in \R^m\times \R_{++}$ is a parameter pair around some fixed $(\bar b,\bar \mu)\in \R^m\times \R_{++}$. The following are equivalent: 
 \begin{itemize}
\item[{\bf(i)}] The following Strong Sufficient Condition holds
 \begin{equation}\label{LipS}
     \Ker \Phi\cap \mathcal{V}=\{0\},
 \end{equation}
 where $\mathcal{V}$ is defined in \eqref{V} with $\oy=-\frac{1}{\bar \mu}\Phi^*(\Phi\ox-\bar b)$.

\item[{\bf(ii)}] The solution mapping $S$ of problem \eqref{Las21} is single-valued and Lipschitz continuous around $(\bar b,\bar \mu)$ with $S(\bar b,\bar \mu)=\ox$.

\item[{\bf(iii)}] The solution mapping $S$ of problem \eqref{Las21} is single-valued  around $(\bar b,\bar \mu)$ with $S(\bar b,\bar \mu)=\ox$.
 \end{itemize}
\end{Theorem}
\begin{proof} According to Corollary~\ref{LipS12}, [{\bf (i)}$\Rightarrow${\bf (ii)}] is valid. The implication [{\bf (ii)}$\Rightarrow${\bf (iii)}] is trivial. It suffices to verify [{\bf (iii)}$\Rightarrow${\bf (i)}]. Suppose that the solution mapping \eqref{S} is single-valued around $(\bar b,\bar \mu)$ with $S(\bar b,\bar \mu)=\ox$. With all the sets $\mathcal{J}, \mathcal{K}, \mathcal{H}$ defined in Corollary~\ref{LasT} and  $\mathcal{I}=\{J\in \mathcal{J}|\; \ox_J\neq 0\}$, we define from \eqref{xy}
\begin{equation}\label{xy2}
x^k=\left\{\begin{array}{ll}(x^k)_J=\ox_J\;\; &\mbox{if}\;\; J\in \mathcal{I}\cup \mathcal{H}\\
(x^k)_J=\frac{1}{k}\oy_J\;\;&\mbox{if}\;\; J\in \mathcal{K}\setminus \mathcal{I}
\end{array}\right.\quad \mbox{and}\quad y^k=\oy. 
\end{equation}
Note that $(x^k,y^k)\to (\ox,\oy)$ and $y^k\in \partial \|x^k\|_{1,2}$. Define 
\[
b_k= \Phi(x^k-\ox)+\bar b.
\]
Observe that 
\[
-\frac{1}{\bar \mu}\Phi^*(\Phi x^k-b_k)=-\frac{1}{\bar \mu} \Phi^*( \Phi \ox-\bar b)=y^k\in \partial \|x^k\|_{1,2}.
\]
It follows that $x^k$ is the solution of the following problem 
\begin{equation}\label{B12}
\min_{x\in\R^n}\quad \frac{1}{2\bar\mu} \|\Phi x-b_k\|^2+\|x\|_{1,2},
\end{equation}
which means $x^k=S(b_k,\bar \mu)$ is the unique solution of problem \eqref{B12} for large $k$. It is proved in \cite[Theorem~5.12]{FNT21} that a unique solution of group Lasso problem \eqref{B12} is actually a strong solution. Hence $x^k$ is also the strong solution of problem \eqref{B12}. By \eqref{Stro}, there exists some $\ell_k>0$ such that 
\[
\ell_k \|w\|^2\le \|\Phi w\|^2+d^2\|\cdot\|_{1,2}(x^k|\, y^k)(w)\quad \mbox{for all}\quad w\in \R^n. 
\]
As the $\ell_{1}/\ell_{2}$ norm satisfies the quadratic growth condition at $x_k$ for $y_k$ in Example~\ref{l12}, we obtain from Lemma~\ref{QGSS} that 
\begin{equation}\label{B21}
\Ker \Phi \cap T_{(\partial \|\cdot\|_{1,2})^{-1}(y^k)}(x^k)=\{0\}.
\end{equation}
Note again from \eqref{V12}  that $
T_{(\partial \|\cdot\|_{1,2})^{-1}(y^k)}(x^k)=\mathcal{V}$. 
This together with \eqref{B21} verifies Strong Sufficient Condition \eqref{LipS}. 
\end{proof}

\begin{Remark}{\rm A particular case of $\ell_{1,2}$ norm is the $\ell_1$ norm at which problem \eqref{Las21} is the Lasso problem. Lipschitz stability of solution mapping of this problem was studied recently in \cite{BBH22}. Our Strong Sufficient Condition \eqref{LipS} is exactly \cite[Assumption~4.3]{BBH22}, which was introduced in \cite{T13} as a sufficient condition for the solution uniqueness of Lasso; see also \cite[Lemma~1]{MY12} for a similar condition. \cite{BBH22} shows that \eqref{LipS} is  only sufficient condition for Lipschitz stability of the solution mapping $S$. We advance \cite[Theorem~4.13]{BBH22} by adding that it is also necessary condition. Moreover, if $S$ is single-valued around the point in question, it is also Lipschitz continuous. \hfill$\triangle$ 
}
\end{Remark}

Strong Sufficient Condition~\eqref{LipS} actually means that the vectors $\{\Phi_J \oy_J\}_{J\in\mathcal{K}}$ are linearly independent, where $\Phi_J$ is the submatrix of $\Phi$ with columns $J$. This condition is stronger than the one that $\{\Phi_J \ox_J\}_{J\in\mathcal{I}}$ is linearly independent introduced in \cite{LZ09} to prove the $\ell^2$-consistency of group Lasso problem. The latter condition is also equivalent to the Restricted Positive Definiteness \cite{VDFPD17} for this framework, which is used to show that the solution mapping $S(\cdot, \bar \mu)$  is continuously differentiable at $\bar b$ under another nontrivial condition that $\bar b$ does not belong to the so-called {\em transition space}; see \cite[Theorem~1]{VDFPD17}. When $S(\cdot, \bar \mu)$ is continuously differentiable around $\bar b$, it should be single-valued, i.e., Strong Sufficient Condition is valid by Theorem~\ref{Lip12} (fixing $\mu$ does not change the proof of Theorem~\ref{Lip12}). The following example shows that the solution mapping $S(\cdot,\bar \mu)$ may be single-valued and Lipschitz continuous without being  differentiable. Of course, one of the reasons is our $\bar b$ does not belong the transition space as in \cite[Theorem~1]{VDFPD17}. As a result, our Strong Sufficient Condition~\eqref{LipS} is weaker than the two aforementioned conditions in \cite[Theorem~1]{VDFPD17}.  

 \begin{Example}[Non-differentiability of solution mapping]\label{Non}{\rm Let us consider the following group Lasso problem ($2$ groups) 
 \begin{equation}
    \mathcal{P}(b)\qquad  \min_{x\in \R^3}\quad \frac{1}{2}\|\Phi x-b\|^2+g(x)
 \end{equation}
 with $g(x)=\sqrt{x_1^2+x_2^2}+|x_3|$, $\Phi=\begin{pmatrix} 1&1&0\\1&0&-1\end{pmatrix}$, $\bar \mu=1$, and $\bar b=(2,-1)^T$. Define $\ox=(0,1,0)^T$ and $\oy=-\Phi^*(\Phi \ox-\bar b)=(0,1,1)^T$. Note that $\oy\in \partial g(\ox)$, which means $\ox$ is an optimal solution of $\mathcal{P}(\bar b)$. The set $\mathcal{V}$ in \eqref{V12} is computed by 
 \[
 \mathcal{V}={\rm span}\{(0,1,0)^T,(0,0,1)^T\}. 
 \]
It is easy to check that Strong Sufficient Condition \eqref{LipS} holds, i.e., $\Ker \Phi\cap \mathcal{V}=\{0\}$. By Theorem~\ref{Lip12}, $x(b)=S(b,1)$ is single-valued and Lipschitz continuous around $\bar b$ with $x(\bar b)=\ox$. Note that $x(b)$ should satisfies the system $-\Phi^*(\Phi x-b)\in \partial g(x)$ when $b$ is close to $\bar b$, which means 
\begin{equation}\label{S3}
    \left\{\begin{array}{ll} -(2x_1+x_2-x_3-b_1-b_2)=\frac{x_1}{\sqrt{x_1^2+x_2^2}}\\
    -(x_1+x_2-b_1)=\frac{x_2}{\sqrt{x_1^2+x_2^2}}\\
    (x_1-x_3-b_2)\in \partial|x_3|
    \end{array}\right.
\end{equation}
As $(x_1-x_3-b_2)$ is close to $1$, as $b$ is around $\bar b$ and $x(b)$ is around $\ox$, the third inclusion  yields $x_3\ge 0$. If $x_3>0$, we have $x_1-x_3-b_2=1$. It follows from the first two equations of \eqref{S3} that 
\[
1=x_1-x_3-b_2=\frac{x_2-x_1}{\sqrt{x_1^2+x_2^2}},
\]
which yields $x_1=0$, $x_2=1$, and that $x_3=-b_2-1> 0$. Hence, if $b_2< -1$, $x(b)=(0,1,-b_2-1)^T$ is the unique solution of system \eqref{S3} by Theorem~\ref{Lip12}. If $b_2\ge -1$, the above arguments  tell us that $x_3=0$. It follows that $x_3(b)=\max\{-b_2-1,0\}$, which is not differentiable at $\bar b$. \hfill$\triangle$  }
 \end{Example}

\begin{Example}[The case of indicator function of the positive semidefinite cone $\mathbb{S}^n_+$] {\rm When $g=\delta_{\mathbb{S}^n_+}$ as in Example~\ref{SDP} and  $(\OX,\OY)\in \gph N_{\mathbb{S}^n_+}$ with the decomposition in \eqref{SD}, we have 
\[
(N_{\mathbb{S}^n_+})^{-1}(\OY)=N_{\mathbb{S}^n_-}(\OY)=\Bar P\begin{pmatrix}\mathbb{S}^{n-|\eta|}_+ &0\\0&0_\eta\end{pmatrix}\Bar P^T. 
\]
It follows that 
\begin{equation*}
    T_{(N_{\mathbb{S}^n_+})^{-1}(\OY)}(\OX)=\left\{\Bar P\begin{pmatrix}A&B&0\\B^T&C&0\\0&0&0\end{pmatrix}\Bar P^T|\; A\in \mathbb{S}^{|\al|}, B\in \R^{|\al|\times |\beta|}, C\in \mathbb{S}^{|\beta|}_+\right\}\subset \Bar P\begin{pmatrix}\mathbb{S}^{n-|\eta|} &0\\0&0_\eta\end{pmatrix}\Bar P^T.
\end{equation*}
The set $\mathfrak{T}_{\mathcal{M}}$ in \eqref{T} for this case of $g$ is easily computed by 
\begin{equation*}\label{TSDP}
  \mathfrak{T}_{\mathcal{M}_\mathbb{S}}\eqdef \Limsup_{(X,Y)\stackrel[Y\in \mathcal{M}_\mathbb{S}]{{\rm gph}\,N_{\mathbb{S}^n_+}}{\longrightarrow}(\OX,\OY)}T_{(N_{\mathbb{S}^n_+})^{-1}(Y)}(X)= \Bar P\begin{pmatrix}\mathbb{S}^{n-|\gg|} &0\\0&0_\eta\end{pmatrix}\Bar P^T. 
\end{equation*}
We skip the detail here. \hfill$\triangle$ 


}
\end{Example}

\begin{Remark}[Perturbation on the linear operator $\Phi$]\label{PPhi}{\rm The linear operator $\Phi$  can be also considered a parameter in problem in \eqref{QP}. Let $\mathscr{L}(\XX,\YY)$ be the space of all linear operators from $\XX$ to $\YY$. The solution set in \eqref{S} with respect to the triple $(\Phi,b,\mu)\in \mathscr{L}(\XX,\YY)\times \R^m\times\R_{+}$ is defined by 
\begin{eqnarray}\label{SPb}
    S(\Phi,b,\mu)={\rm argmin}\left\{\frac{1}{2\mu}\|\Phi x-b\|^2+g(x)|\; x\in \XX\right\}=\left\{x\in \XX|\; 0\in \frac{1}{\mu}\Phi^*(\Phi x-b)+\partial g(x)\right\}.
\end{eqnarray}
It is similar to prove that this solution mapping  is single-valued and Lipschitz continuous around some known triple $(\Bar \Phi,\bar b,\mu) \in \mathscr{L}(\XX,\YY)\times \R^m\times\R_{++}$ provided that the 
\[
\Ker \Bar\Phi \cap \mathfrak{T}_\mathcal{M}=\{0\}\quad \mbox{ with}\quad \oy= -\frac{1}{\bar \mu}\Bar\Phi^*(\Bar\Phi \ox-\bar b)
\]
in the settings of Corollary~\ref{LipS12}. When $g(\cdot)=\|\cdot\|_{1,2}$, the Lipschitz stability of this mapping is equivalent to the condition \eqref{LipS}
\[
\Ker \Bar\Phi\cap\mathcal{V}=\{0\}\quad \mbox{ with} \quad \oy= -\frac{1}{\bar \mu}\Bar\Phi^*(\Bar\Phi \ox-\bar b).
\]
}    
\end{Remark}

\section{Geometric characterizations of Lipschitz stability for nuclear norm minimization problems}
\setcounter{equation}{0}
In this section, we consider the following nuclear norm minimization problem
\begin{equation}\label{P1}
\min_{X\in \R^{n_1\times n_2}}\qquad f(X)+\|X\|_*,
\end{equation}
where $f:\R^{n_1\times n_2}\to \R$ is a twice continuously differentiable convex function and $g(X)\eqdef\|X\|_*$ is the nuclear norm of $X\in \R^{n_1\times n_2}$, which is the sum of its singular values $\sigma_1(X), \ldots, \sigma_{n_1}(X)$ with nonincreasing order. For any $\OX\in \R^{n_1\times n_2}$ ($n_1\le n_2$), the singular valued decomposition (SVD) of $\OX$ is written by
\[
\OX=U\begin{pmatrix}\Sigma &0\\0&0\end{pmatrix}_{n_1\times n_2}V^T,
\]
where $U\in \R^{n_1\times n_1}$ and $V\in \R^{n_2\times n_2}$ are orthogonal matrices, $\Sigma={\rm Diag}\, \{\sigma_1,\sigma_2, \ldots,\sigma_r\}$ is the $r\times r$ diagonal matrix with $\sigma_1\ge \sigma_2\ldots\ge \sigma_r>0$ being all positive singular values of $\OX$, and $r$ is the rank of $\OX$. The set $\mathcal{O}(\OX)$ contains all such pairs $(U,V)$ for $\OX$.

The following lemma provides some well-known formulas for subdifferential of the nuclear norm; see, e.g., \cite[Example~1 and Example~2]{W92} and \cite[Proposition 10]{ZS17}. Recall that $\|X\|$ is the spectral norm of $X\in \R^{n_1\times n_2}$, which is the largest singular value of $X$. 

\begin{Lemma}[Subdifferential of the nuclear norm]\label{Lem} The subdifferential to nuclear norm at $\OX\in \R^{n_1\times n_2}$ is computed by
\begin{equation}\label{subdif}
    \partial\|\OX\|_*=\left\{U\begin{pmatrix} \Id_r&0\\ 0 & W\end{pmatrix}V^T|\; \|W\|\le 1\right\} \quad \mbox{for any}\quad (U,V)\in \mathcal{O}(\OX). 
\end{equation}
Moreover, $\OY\in \partial\|\OX\|_*$ if and only if $\|\OY\|\le 1$ and 
\begin{equation}\label{Fen}
    \|\OX\|_*=\la \OY,\OX\ra. 
\end{equation}
Furthermore, for any $\OY\in\B\eqdef\{Z\in \R^{n_1\times n_2}|\;\|Z\|\le 1\}$, we have 
\begin{equation}\label{Inver}
\partial g^*(\OY)= N_\B(\OY)=\OU\begin{pmatrix}\mathbb{S}_+^{p(\OY)} &0\\0 &0\end{pmatrix}\OV^T\quad \mbox{for any}\quad (\OU,\OV)\in \mathcal{O}(\OY),
\end{equation}
where $p(\OY)$ is defined by
\begin{equation}\label{p}
p(\OY)\eqdef\#\{i\in \{1, \ldots,n_1\}|\; \sigma_i(\OY)=1\}.
\end{equation}
\end{Lemma}

When $\OY\in \partial \|\OX\|_*$, $\OX$ and $\OY$ have the {\em simultaneous} ordered singular value decomposition in the sense that there exists  $(\OU,\OV)\in \mathcal{O}(\OX)\cap\mathcal{O}(\OY)$ such that
\begin{equation}\label{SVD}
\OX=\OU\begin{pmatrix}\Sigma &0\\0&0\end{pmatrix}\OV^T\qquad \mbox{and}\qquad \OY=\OU\begin{pmatrix}\Id_p&0&0\\0&\Lm&0\end{pmatrix}\OV^T,
\end{equation}
where $p=p(\OY)$ and $\Lm\in \R^{(n_1-p)\times (n_1-p)}$ is the diagonal matrix containing all singular values of $\OY$ in $[0,1)$; 
see, e.g., \cite{FNP23} for the way of constructing the pair $(\OU,\OV)\in \mathcal{O}(\OX)$. In the following result, we show that the nuclear norm satisfies the quadratic growth condition for  some nontrivial set $\mathcal{M}$. Local quadratic growth condition of the nuclear norm was established in \cite{ZS17} without providing the exact formula for the modulus. Our approach here is completely different via convex analysis. It reveals not only the modulus but also the {\em global}  quadratic growth condition of the nuclear norm.

\begin{Lemma}
    [Quadratic growth condition for  a set of the nuclear norm]\label{QG} Let $(\OX,\OY)\in \gph \partial \|\cdot\|_*$. Then we have
\begin{eqnarray}\label{Gr1}
    \|X\|_*-\|\OX\|_*-\la \Bar Y,X-\OX\ra\ge \frac{1-\gg^2}{2\|X\|_*(1+(1+\gg)^2)}\left[{\rm dist}\,(X; (\partial\|\cdot\|_*)^{-1}(\Bar Y))\right]^2
\end{eqnarray}
for any  $X\in \R^{n_1\times n_2}\setminus\{0\}$,
where $\gg$ is the {\em subdominant singular value} defined by
\begin{equation}\label{subd}
    \gg=\gg(\OY)\eqdef\max\{\sigma(\Bar Y)\setminus\{1\},0\}<1.
\end{equation}
Consequently, we have 
\begin{equation}\label{GrI}
 \|X\|_*-\|\OX\|_*-\la \Bar Y,X-\OX\ra\ge \frac{1-\gg}{5\|X\|_*}\left[{\rm dist}\,(X; (\partial\|\cdot\|_*)^{-1}(\Bar Y))\right]^2\quad \mbox{for all}\quad X\in \R^{n_1\times n_2}\setminus\{0\}.
\end{equation}
Finally, the nuclear norm satisfies the quadratic growth condition around $(\OX,\OY)$ for  the following set 
\begin{equation}\label{M}
\mathcal{M}_*\eqdef\{Y\in \R^{n_1\times n_2}|\; \sigma_1(Y)=\ldots=\sigma_p(Y)=1\},
\end{equation}
where $p=p(\OY)$ is the number of singular values of $\OY$ that are equal to $1$ defined in \eqref{p}. 
\end{Lemma}
\begin{proof} Suppose that $U\Bar \Sigma V^T$ is the SVD of $\Bar Y$ with $\Bar \Sigma=\begin{pmatrix}\Id_p&0\\0&\Sigma\end{pmatrix}_{n_1\times n_2}$,  $(U,V)\in \mathcal{O}(\OY)$, and   $\|\Sigma\|<1$, where $p$ is from \eqref{p}. 
Pick any $X\in \R^{n_1\times n_2}\setminus\{0\}$ and define $\Hat X=\|X\|_*^{-1}X$. As $(\partial \|\cdot\|_*)^{-1}(\Bar Y)$ is a convex cone by \eqref{Inver} and $\|\OX\|_*=\la \OY,\OX\ra$ by \eqref{Fen}, inequality \eqref{Gr1} is equivalent to
\begin{equation}\label{Gr3}
\|\Hat X\|_*-\la \Bar Y,\Hat X\ra\ge \frac{c}{2}\left[{\rm dist}\,(\Hat X;(\partial \|\cdot\|_*)^{-1}(\Bar Y))\right]^2\quad \mbox{with}\quad c\eqdef \frac{(1-\gg^2)}{1+(1+\gg)^2}.
\end{equation}
Define $\Tilde X=U^T\Hat X V$ and note that 
\[
\|\Hat X\|_*-\la \Bar Y,\Hat X\ra=\|\Tilde X\|_*-{\rm Tr}(V\Bar \Sigma^T U^T \Hat X)=\|\Tilde X\|_*-{\rm Tr}(\Bar \Sigma^T U^T \Hat  XV)=\|\Tilde X\|_*-\la \Bar \Sigma,\Tilde X\ra. 
\]
Moreover, it follows from \eqref{Inver} that 
\[
{\rm dist}\, (\Hat X;(\partial \|\cdot\|_*)^{-1}(\Bar Y))={\rm dist}\,\left(\Tilde X;\begin{pmatrix}\mathbb{S}^p_+&0\\0&0\end{pmatrix}\right).
\]
Thus, to prove \eqref{Gr1}, we only need to show that 
\begin{equation}\label{Gr2}
    h(X)\eqdef\|X\|_*-\la \Bar \Sigma,X\ra+\delta_{B^*}(X)\ge k(X)\eqdef\frac{c}{2} \left[{\rm dist}\left(X; \begin{pmatrix}{\mathbb S}^p_+&0\\0&0\end{pmatrix}\right)\right]^2\quad \mbox{for}\quad X\in \R^{n_1\times n_2},
\end{equation}
where $B_*$ is the unit ball with the nuclear norm, i.e., ${\mathbb B}_*\eqdef\{X\in \R^{n_1\times n_2}|\; \|X\|_*\le 1\}$. By the Biconjugate Theorem of convex analysis, inequality \eqref{Gr2} is equivalent to 
\begin{equation}\label{Fe2}
h^*(Y)\le k^*(Y)\quad \mbox{for all}\quad Y\in \R^{n_1\times n_2}.
\end{equation}
Note that 
\begin{eqnarray}\label{Y1}\begin{array}{ll}
h^*(Y)&=\disp\sup_{X\in {\mathbb B}_*}\la Y+\Bar\Sigma,X\ra-\|X\|_*
=\max\{\|Y+\Bar\Sigma\|- 1,0\}=\disp(\|Y+\Bar\Sigma\|- 1)_+.
\end{array}
\end{eqnarray}
By writing $Y=\begin{pmatrix}A&B\\C&D\end{pmatrix}\in \R^{n_1\times n_2}$ as a block matrix with $A\in \R^{p\times p}$, we have
\begin{eqnarray}\label{Y2}\begin{array}{ll}
k^*(Y)&=\disp\sup_{X\in \R^{n_1\times n_2}}\sup_{Z\in {\mathbb S}^p_+}\la Y,X\ra-\frac{c}{2}\left\|X-\begin{pmatrix}Z&0\\0&0\end{pmatrix}\right\|_F^2\\
&=\disp\sup_{Z\in {\mathbb S}^p_+}\sup_{X\in \R^{m\times n}} \left\la A,Z\right\ra+\la Y, X-\begin{pmatrix}Z&0\\0&0\end{pmatrix}\ra-\frac{c}{2}\left\|X-\begin{pmatrix}Z&0\\0&0\end{pmatrix}\right\|^2_F\\
&=\disp\sup_{Z\in {\mathbb S}^p_+}\la A,Z\ra+\dfrac{1}{2c}\|Y\|^2_F\\
&=\disp\sup_{Z\in {\mathbb S}^p_+}\frac{1}{2}\la A+A^T,Z\ra+\dfrac{1}{2c}\|Y\|^2_F\\
&=\disp\delta_{{\mathbb S}^p_-}(A+A^T)+\dfrac{1}{2c}\|Y\|^2_F, 
\end{array}\end{eqnarray}
where the last equality holds due to $A+A^T\in \mathbb{S}^p$. By \eqref{Y1} and \eqref{Y2}, we only need to verify \eqref{Fe2} when $A+A^T\in {\mathbb S}^p_-$. Recall that for a symmetric matrix $Z\in \R^{n_2\times n_2}$, we write  $Z\succeq0$ if $Z\in \mathbb{S}^{n_2}_+$ and $Z\preceq0$ if $-Z\succeq0$. Note from the representation $\Bar\Sigma=\begin{pmatrix}\Id_p&0\\0&\Sigma\end{pmatrix}$ that 
\begin{eqnarray}\label{Eq1}\begin{array}{ll}
(Y+\Bar \Sigma)^T(Y+\Bar \Sigma)&=Y^TY+\Bar\Sigma^TY+Y^T\Bar\Sigma+\Bar\Sigma^T\Bar\Sigma\\
&=Y^TY+\begin{pmatrix}A+A^T&B+C^T\Sigma\\B^T+\Sigma^TC&\Sigma^TD+D^T\Sigma\end{pmatrix}+\begin{pmatrix}\Id_p&0\\0&\Sigma^T\Sigma\end{pmatrix}\\
&\preceq(\|Y\|^2+1)\Id_{n_2}+\begin{pmatrix}0&B+\Sigma^TC \\B^T+C^T\Sigma&\Sigma^TD+D^T\Sigma+\Sigma^T\Sigma-\Id_{n_2-p} \end{pmatrix}.
\end{array}
\end{eqnarray}
Next we claim that 
\begin{equation}\label{Eq2}
\begin{pmatrix}0&B+\Sigma^TC \\B^T+C^T\Sigma&\Sigma^TD+D^T\Sigma+\Sigma^T\Sigma-\Id_{n_2-p} \end{pmatrix}\preceq d\|Y\|^2\Id_{n_2}\quad \mbox{with}\quad d\eqdef\dfrac{\gg^2+(1+\gg)^2}{(1-\gg^2)},
\end{equation}
which is equivalent to 
\[
\begin{pmatrix}d\|Y\|^2\Id_p&-B-\Sigma^TC \\-B^T-C^T\Sigma&(d\|Y\|^2+1)\Id_{n_2-p}-\Sigma^TD-D^T\Sigma-\Sigma^T\Sigma \end{pmatrix}\succeq 0.
\]
If $\|Y\|=0$, the latter is trivial as $\Id_{n_2-p}\succeq\Sigma^T\Sigma$ and $D=0$. If $\|Y\|>0$, by Schur complement, the above inequality means \begin{equation}\label{Eq3}
(d\|Y\|^2+1)\Id_{n_2-p}-\Sigma^TD-D^T\Sigma-\Sigma^T\Sigma-\frac{1}{d\|Y\|^2}(B^T+C^T\Sigma)(B+\Sigma^TC)\succeq 0. 
\end{equation}
Note that  $\gg=\|\Sigma\|$ and   $\max\{\|B\|,\|C\|,\|D\|\}\le\|Y\|$. As $-\Sigma^TD-D^T\Sigma$ is symmetric, we have 
\[
-\Sigma^TD-D^T\Sigma\succeq -\|\Sigma^TD+D^T\Sigma\|\Id_{n_2-p}\succeq -2\|\Sigma^TD\|\Id_{n_2-p} \ge -2\|\Sigma^T\|\cdot\|D\|
\Id_{n_2-p}\succeq-2\gg\|Y\|\Id_{n_2-p}. \]
Similarly, 
\[
-(B^T+C^T\Sigma)(B+\Sigma^TC)\succeq -\|B^T+C^T\Sigma\|^2\Id_{n_2-p}\succeq-(\|B\|+\|C\|\cdot\|\Sigma\|)^2\Id_{n_2-p}\succeq-(1+\gg)^2\|Y\|^2\Id_{n_2-p}. 
\]
These two inequalities imply that
the left-hand side of \eqref{Eq3} is  ``$\succeq$'' than 
\[\begin{array}{ll}
 \left[(d\|Y\|^2+1)
-2\gg\|Y\|-\gg^2-\frac{(1+\gg)^2}{d}\right]\Id_{n_2-p}&=\left[d\|Y\|^2-2\gg\|Y\|+\left(1-\gg^2-\frac{(1+\gg)^2}{d}\right)\right]\Id_{n_2-p}\\
&\succeq 2\left[\sqrt{d\left(1-\gg^2-\frac{(1+\gg)^2}{d}\right)}-\gg)\right]\|Y\|\Id_{n_2-p}=0,
\end{array}
\]
where the last equality holds due to the choice of $d=\dfrac{\gg^2+(1+\gg)^2}{(1-\gg^2)}$. This verifies \eqref{Eq3} and thus \eqref{Eq2}. It follows from \eqref{Eq1} and \eqref{Eq2} that 
\[
(Y+\Bar \Sigma)^T(Y+\Bar \Sigma)\preceq ((1+d)\|Y\|^2+1)\Id_{n_2},
\]
which means $\|Y+\Bar \Sigma\|\le \sqrt{(1+d)\|Y\|^2+1}$. Hence we have
\[
\|Y+\Bar \Sigma\|-1\le \dfrac{(1+d)\|Y\|^2}{\sqrt{(1+d)\|Y\|^2+1}+1}\le \dfrac{(1+d)\|Y\|^2}{2}=\dfrac{(1+(1+\gg)^2)}{2(1-\gg^2)}\|Y\|_F^2.
\]
This together with \eqref{Y1} and \eqref{Y2} verifies \eqref{Fe2}, \eqref{Gr2}, and thus \eqref{Gr1}. 

To justify \eqref{GrI}, we notice that 
\[
\frac{1-\gg^2}{2(1+(1+\gg)^2)}=\frac{1-\gg}{2}\cdot\frac{1+\gg}{1+(1+\gg)^2}\ge \frac{1-\gg}{2}\cdot \frac{2}{5}=\frac{1-\gg}{5},
\]
as $\gg\in [0,1)$. Combining this  with \eqref{Gr1} gives us \eqref{GrI}.

Finally, we prove the quadratic growth condition of $\|\cdot\|_*$ around $(\OX,\OY)$ for  the set $\mathcal{M}_*$ defined in \eqref{M}.  Pick any $\ve>0$ with $\ve<1-\gg$ and take any $(X,Y)\in \gph \partial \|\cdot\|_*\cap \B_\ve(\OX,\OY)$ with $Y\in \mathcal{M}_*$. It is well-known that $\sigma_{p+1}(\cdot)$ is Lipschitz continuous with modulus $1$ (Weyl's inequality). We have 
\[
\sigma_{p+1}(Y)\le \sigma_{p+1}(\OY)+\|Y-\OY\|_F\le\gg+\ve<1,
\]
which tells us that $p(Y)=p(\OY)$. For any $U\in \B_\ve(X)\subset \B_{2\ve}(\OX)$,  it follows from \eqref{GrI} that 
\begin{equation*}\begin{array}{ll}
 \|U\|_*-\|X\|-\la  Y,U-X\ra&\disp\ge \frac{1-\gg(Y)}{5\|U\|_*}\left[{\rm dist}\,(U; (\partial\|\cdot\|_*)^{-1}( Y))\right]^2\\
 &\disp\ge \frac{1-\gg-\ve}{5(\|\OX\|_*+2\ve)}\left[{\rm dist}\,(U; (\partial\|\cdot\|_*)^{-1}( Y))\right]^2.
 \end{array}
\end{equation*}
As $\ve>0$ can be chosen sufficiently small, $\|\cdot\|_*$ satisfies the quadratic growth condition around $(\OX,\OY)$ for  $\mathcal{M}_*$ with some constant modulus .  
\end{proof}

\begin{Remark}[Improving the quadratic growth condition \eqref{GrI}]{\rm The subdominant singular value $\gg$ is indeed the second largest singular value of $\OY$ when $0<p<n_1$. If $p=0$, $\gg=\|\OY\|$ and if $p=n_1$, $\gg=0$. In \eqref{GrI}, the term $\frac{1-\gg}{5}$ is not sharp; see, e.g., \eqref{Gr1} for a better constant. It is interesting to question the best bound for this inequality. By looking back the quadratic growth condition \eqref{QG5}  for $\ell_1/\ell_2$ norm, we conjecture that \eqref{GrI} can be sharpened to 
\begin{equation*}
 \|X\|_*-\|\OX\|_*-\la \Bar Y,X-\OX\ra\ge \frac{1-\gg}{2\|X\|_*}\left[{\rm dist}\,(X; (\partial\|\cdot\|_*)^{-1}(\Bar Y))\right]^2\quad \mbox{for all}\quad X\in \R^{n_1\times n_2}\setminus\{0\}.
\end{equation*}
At the moment, we only have affirmative answers for $X\in \R^{2\times 2}$ and some special cases of $\OY$ in higher dimensions.  \hfill$\triangle$}
\end{Remark}

To apply Theorem~\ref{GeoTilt}, we need to show that the subdifferential mapping $\partial \|\cdot\|_*$ has relative approximations  by $\mathcal{M}_*$.

\begin{Proposition}[Relative approximations of subgradients of the nuclear norm]\label{RelA*} Let $\OY\in \partial \|\OX\|_*$. Then the subdifferential mapping $\partial \|\cdot\|_*$ has relative approximations  by the set $\mathcal{M}_*$ defined in \eqref{M}. 
\end{Proposition}
\begin{proof} Pick any $(X,Y)\in \gph \partial \|\cdot\|_*\cap \B_\ve(\OX,\OY)$ with $\ve>0$.  Since $\|Y\|\le 1$ and $\sigma_k(Y)$ are Lipschitz continuous for any $k$, we have $
{\rm rank}\, (\OX)\le {\rm rank}\, (X)\le p(Y)\le p(\OY)$
when $\ve$ is sufficiently small, where the $p(Y)$ counting how many singular values of $Y$ are equal to $1$; see \eqref{p}. 
If $p(Y)= p(\OY)$, i.e., $Y\in \mathcal{M}_*$, a relative approximation of $Y$  by $\mathcal{M}_*$ is $Y$. If $ q=p(Y)< p(\OY)$, suppose that a SVD of $Y$ is 
\[
Y=U\begin{pmatrix}\Id_q&0&0\\0&\Sigma&0\end{pmatrix} V^T,
\]
where $\Sigma={\rm diag}(\sigma_{q+1}, \ldots, \sigma_{n_1})$ with $\sigma_1\ge \ldots\ge\sigma_p\ge \ldots \sigma_{n_1}$ being all singular values of $Y$ and $(U,V)\in \mathcal{O}(X)\cap\mathcal{O}(Y)$. Let us set 
\[
\Hat Y=U\begin{pmatrix}\Id_p&0&0\\0&\Hat \Lm&0\end{pmatrix} V^T\qquad \mbox{and}\qquad \Tilde Y=U\begin{pmatrix}\Id_q&0&0\\0&\Tilde\Lm&0\end{pmatrix} V^T,
\]
where $\Hat \Lm$ and $\Tilde\Lm$ are two diagonal matrices defined by
\[
\Hat \Lm={\rm Diag}\, (\sigma_{p+1}, \ldots, \sigma_{n_1})\qquad \mbox{and}\qquad \Tilde \Lm=\Diag\left(\frac{\sigma_{q+1}-\sigma_{p}}{1-\sigma_{p}}, \ldots,\frac{\sigma_{p-1}-\sigma_{p}}{1-\sigma_{q}},0, \sigma_{p+1}, \ldots, \sigma_{n_1}\right).
\]
As $r\le q$, observe that both $\Hat Y, \Tilde Y\in \partial \|X\|$ by Lemma~\ref{Lem}. Moreover, $\Hat Y\in \mathcal{M}_*$ and $
Y=\sigma_p \Hat Y+(1-\sigma_p)\Tilde Y$. 
As $\sigma_p(\cdot)$ is Lipschitz continuous with modulus $1$, we have
\[
1-\sigma_p=|1-\sigma_p|=|\sigma_p(\OY)-\sigma_p(Y)|\le \|\OY-Y\|_F\le \ve,
\]
which tells us that $\sigma_p\in (1-\ve,1)$.
Hence $\Hat Y$ is a relative approximation of $Y$  by $\mathcal{M}_*$. Thus  $\partial \|\cdot\|_*$ has relative approximations  by the set $\mathcal{M}_*$.\end{proof}

We recall here the computation of the tangent cone $T_{(\partial \|\cdot\|_*)^{-1}(\OY)}(\OX)$ in the recent paper \cite[Corollary~4.2]{FNP23}.

\begin{Lemma}[Tangent cone] Suppose that $\OX,\OY$ have representations as in \eqref{SVD}. We have
\begin{equation}\label{Tang}
T_{(\partial \|\cdot\|_*)^{-1}(\OY)}(\OX)=\left\{\OU\begin{pmatrix} A&B&0\\B^T&C&0\\0&0&0\end{pmatrix}\OV^T|\; A\in \mathbb{S}^r, B\in \R^{r\times(p-r)}, C\in \mathbb{S}^{(p-r)\times(p-r)}_+ \right\}, 
\end{equation}
where $r={\rm rank}\,(\OX)$ and $p=p(\OY)$ defined in \eqref{p}. This formula does not depend on the choice of $(\OU,\OV)\in \mathcal{O}(\OX)\cap\mathcal{O}(\OY)$.
\end{Lemma}

We are ready to establish the main result in this section, which gives a complete characterization of tilt stability of problem \eqref{Gr1}.

\begin{Theorem}\label{GeoTilt*} Suppose that $\OX$ is an optimal solution of problem \eqref{P1}, i.e., $\OY=-\nabla f(\OX)\in \partial\|\OX\|_*$. Hence $\OX$ is a tilt-stable solution of problem \eqref{P1} if and only if the Strong Sufficient Condition holds:
\begin{equation}\label{SSC}
\Ker \nabla^2 f(\OX)\cap \OU\begin{pmatrix}\mathbb{S}^p&0\\0&0\end{pmatrix}\OV^T=\{0\} \qquad\mbox{with}\qquad p=p(\OY),
\end{equation}
where   $(\OU,\OV)\in \mathcal{O}(\OX)\cap \mathcal{O}(\OY)$ are from the  simultaneous ordered SVD of $\OX$ and $\OY$ in \eqref{SVD}.
\end{Theorem}
\begin{proof} According Theorem~\ref{QG}, Proposition~\ref{RelA*}, and Theorem~\ref{GeoTilt}, $\OX$ is a tilt-stable solution if and only if condition \eqref{Geo} is satisfied. We claim that 
\begin{equation}\label{LimTang}    \mathfrak{T}_{\mathcal{M}_*}=\Limsup_{(X,Y)\stackrel[Y\in \mathcal{M}_*]{{\rm gph}\,\partial \|\cdot\|_*}{\longrightarrow}(\OX,\OY)} T_{(\partial \|\cdot\|_*)^{-1}(Y)}(X)=\OU\begin{pmatrix}\mathbb{S}^p&0\\0&0\end{pmatrix}\OV^T.  
\end{equation}
Let us verify the ``$\subset$" part first. Pick any $W$ in the left-hand side set. We find sequences $(X_k,Y_k)\in \gph (\partial \|\cdot\|_*)$ and $W_k\in T_{(\partial \|\cdot\|_*)^{-1}(Y_k)}(X_k)$ such that $Y_k\in \mathcal{M}_*$, $(X_k,Y_k)\to (\OX,\OY)$, and $W_k\to W$. As $Y_k\in \partial \|X_k\|_*$, $X_k$ and $Y_k$ have simultaneous ordered SVD, i.e., there exists $(U_k,V_k)\in \mathcal{O}(X_k)\cap \mathcal{O}(Y_k)$ such that  the SVD of $X_k$ and $Y_k$ are \begin{equation}\label{XY}    
X_k=U_k\begin{pmatrix}\Sigma_k&0\\0&0\end{pmatrix}V_k^T\quad \mbox{and}\quad Y_k=U_k\begin{pmatrix}\Id_{p}&0&0\\0&\Lm_k&0\end{pmatrix}V_k^T.
\end{equation}
By passing to subsequences, we may suppose that $(U_k,V_k)\to (U,V)\in \mathcal{O}(\OX)\cap \mathcal{O}(\OY)$, $\Sigma_k\to \Sigma$, and $\Lm_k\to \Lm$. Hence, we have 
\begin{equation}\label{XY2}    
\OX=U\begin{pmatrix}\Sigma&0\\0&0\end{pmatrix}V^T\quad \mbox{and}\quad \OY=U\begin{pmatrix}\Id_{p}&0&0\\0&\Lm&0\end{pmatrix}V^T,
\end{equation}
which tells us that $(U,V)\in \mathcal{O}(\OX)\cap\mathcal{O}(\OY)$. 
As $Y_k\to \OY$ and $p=p(\OY)$,  $\|\Lm\|<1$. By \eqref{Tang}, we have
$
W_k\in U_k\begin{pmatrix}\mathbb{S}^p&0\\0&0\end{pmatrix}V_k^T,
$
which  implies that 
$
W\in U\begin{pmatrix}\mathbb{S}^p&0\\0&0\end{pmatrix}V^T.
$  It follows that 
\begin{equation}\label{eq:TMS}
\mathfrak{T}_{\mathcal{M}_*}\subset \left\{ U\begin{pmatrix}\mathbb{S}^p&0\\0&0\end{pmatrix} V^T|\,  (U, V)\in \mathcal{O}(\Bar X)\cap\mathcal{O}(\Bar Y)\right\}. 
\end{equation}

Recall here that  the formula of the tangent cone $T_{(\partial \|\cdot\|_*)^{-1}(\OY)}(\OX)$  in \eqref{Tang} does not depend on the choice of any $(\Bar U, \Bar V)\in \mathcal{O}(\OX)\cap\mathcal{O}(\OY)$.
As $(U,V)\in \mathcal{O}(\OX)\cap\mathcal{O}(\OY)$, we obtain from \eqref{Tang} that 
\[
T_{(\partial \|\cdot\|_*)^{-1}(\OY)}(\OX)=\left\{U\begin{pmatrix} A&B&0\\B^T&C&0\\0&0&0\end{pmatrix}V^T|\; A\in \mathbb{S}^r, B\in \R^{r\times(p-r)}, C\in \mathbb{S}^{(p-r)\times(p-r)}_+ \right\}.
\]
This gives us that 
\[
\span \left[T_{(\partial \|\cdot\|_*)^{-1}(\OY)}(\OX)\right]=T_{(\partial \|\cdot\|_*)^{-1}(\OY)}(\OX)-T_{(\partial \|\cdot\|_*)^{-1}(\OY)}(\OX)=U\begin{pmatrix}\mathbb{S}^p&0\\0&0\end{pmatrix}V^T
\]
due to the fact that 
\[
\span \mathbb{S}^{(p-r)\times(p-r)}_+=\mathbb{S}^{(p-r)\times(p-r)}_+-\mathbb{S}^{(p-r)\times(p-r)}_+=\mathbb{S}^{(p-r)\times(p-r)}.
\]
Similarly, we also have 
\[
\span \left[T_{(\partial \|\cdot\|_*)^{-1}(\OY)}(\OX)\right]=\OU\begin{pmatrix}\mathbb{S}^p&0\\0&0\end{pmatrix}\OV^T.
\]
It follows that 
\[
U\begin{pmatrix}\mathbb{S}^p&0\\0&0\end{pmatrix}V^T=\OU\begin{pmatrix}\mathbb{S}^p&0\\0&0\end{pmatrix}\OV^T.
\]
This together with \eqref{eq:TMS} verifies the inclusion ``$\subset$'' in \eqref{LimTang}.

To justify the opposite inclusion in \eqref{LimTang},  define 
\begin{equation*}
\OX_k=\OU\begin{pmatrix}\Sigma&0&0\\0&\frac{1}{k}\Id_{p-r}&0\\
0&0&0\end{pmatrix}\OV^T\quad \mbox{and}\quad \OY_k=\OY.
\end{equation*}
Observe that $(\OX_k,\OY_k)\to(\OX,\OY)$ and $\OY_k\in \partial \|\OX_k\|_*\cap\mathcal{M}_*$. Note further from \eqref{Tang} that 
\[
T_{\partial g^*(\OY_k)}(\OX_k)= \OU\begin{pmatrix}\mathbb{S}^p&0\\0&0\end{pmatrix}\OV^T. 
\]
It follows that 
\[
\OU\begin{pmatrix}\mathbb{S}^p&0\\0&0\end{pmatrix}\OV^T\subset \Limsup_{k\to\infty}T_{(\partial \|\cdot\|_*)^{-1}(\OY_k)}(\OX_k)\subset \mathfrak{T}_{\mathcal{M}_*},
\]
which ensures the inclusion ``$\supset$'' in \eqref{LimTang}. Thus equality \eqref{LimTang} holds. By Theorem~\ref{GeoTilt}, $\OX$ is tilt-stable if and only if  Strong Sufficient Condition \eqref{SSC} is satisfied.
\end{proof}

Strong Sufficient Condition~\eqref{SSC} was introduced in \cite[Corollary~4.2]{FNP23} as a sufficient condition for strong  solutions at $\OX$; see \cite[Remark~4.2]{FNP23} for further discussions of this condition and its  relations with some other conditions such as the Restricted Injectivity and Nondegeneracy Condition \cite{CP10,CR09,CR13} for the nuclear norm minimization problem \eqref{P1}. Similarly to Theorem~\ref{Lip12}, we show next that Strong Sufficient Condition \eqref{SSC} is equivalent to the single-valuedness and Lipschitz continuity of solution mapping $S$ in \eqref{S} for the corresponding  parametric nuclear norm minimization problem \eqref{NNM0}. 

\begin{Theorem}\label{LipNu} Consider the following perturbed optimization problem 
\begin{equation}\label{NNM}
{\mathcal P}(B,\mu)\qquad \min_{X\in \R^{n_1\times n_2}}\quad \frac{1}{2\mu}\|\Phi X-B\|^2+\|X\|_*,
\end{equation}
where $\Phi:\R^{n_1\times n_2}\to \R^m$ is a linear operator and the parameter pair$(B,\mu)\in \R^m\times \R_{++}$ is around a known pair  $(\Bar B,\bar \mu)\in \R^m\times \R_{++}$. The following are equivalent: 
\begin{itemize}
\item[{\bf (i)}]  Strong Sufficient Condition holds:
\begin{equation}\label{SSC2}
\Ker \Phi\cap \OU\begin{pmatrix}\mathbb{S}^p&0\\0&0\end{pmatrix}\OV^T=\{0\} \qquad\mbox{with}\qquad \OY=-\frac{1}{\bar \mu}\Phi^*(\Phi \OX-\Bar B)\quad \mbox{and}\quad p=p(\OY).
\end{equation}

\item[{\bf (ii)}] The corresponding solution mapping $S$ in \eqref{S}   is single-valued  and Lipschitz continuous around $(\Bar B,\bar\mu)$ with $S(\Bar B,\bar \mu)=\OX$.

\item[{\bf (iii)}] The  solution mapping $S$ in \eqref{S}   is single-valued  around $(\Bar B,\bar\mu)$ with $S(\Bar B,\bar \mu)=\OX$.
    \end{itemize}

\end{Theorem}
\begin{proof} By Corollary~\ref{LipS12} and Theorem~\ref{GeoTilt*}, Strong Sufficient Condition \eqref{SSC2} implies the Lipschitz continuity of the solution mapping $S$.
 This clearly verifies the implication [{\bf(i)}$\Longrightarrow${\bf(ii)}]. Moreover,  [{\bf(ii)}$\Longrightarrow${\bf(iii)}] is trivial.

It remains to prove [{\bf(iii)}$\Longrightarrow${\bf(i)}]. Let us suppose that $S$ in \eqref{S}   is single-valued around $(\Bar B,\bar \mu)$ with $S(\Bar B,\bar \mu)=\OX$. By the SVD \eqref{SVD} for $\OX$ and $\OY$, we define 
\begin{equation*}
X_k=\OU\begin{pmatrix}\Sigma&0&0\\0&\frac{1}{k}\Id_{p-r}&0\\
0&0&0\end{pmatrix}\OV^T\qquad \mbox{and}\qquad Y_k=\OY\quad {\mbox with}\quad (\OU,\OV)\in \mathcal{O}(\OX)\cap\mathcal{O}(\OY).
\end{equation*}
It follows from Lemma~\ref{Lem} that $Y_k\in \partial\|X_k\|_*$ and $(X_k,Y_k)\to (\OX,\OY)$ as $k\to \infty$. Define further
\[
B_k=\frac{1}{\bar \mu} \Phi(X_k-\OX)+\Bar B\to \Bar B. 
\]
Note that 
\[
-\frac{1}{\bar \mu}\Phi^*(\Phi X_k-B_k)=-\frac{1}{\bar \mu}\Phi^*(\Phi \OX-\Bar B)=\OY\in \|X_k\|_*.
\]
Hence $X_k$ is the optimal solution of the following problem
\begin{equation}\label{NNM2}
\min_{X\in \R^{m\times n}}\quad \frac{1}{2\bar\mu}\|\Phi X-B_k\|^2+\|X\|_*,
\end{equation}
which means $X_k=S(B_k, \bar \mu)$ is the unique solution of problem $\mathcal{P}(B_k,\mu_k)$ in \eqref{NNM2}. We claim that 
\begin{equation}\label{Rad}
\Ker \Phi\cap \R_+ \big[(\partial \|\cdot\|_*)^{-1}(Y_k)-X_k\big]=\{0\}.
\end{equation}
Indeed, pick any $W\in  \Ker \Phi\cap \R_+\big[(\partial \|\cdot\|_*)^{-1}(Y_k)-X_k\big]$. There exists  $\lm>0$ such that 
\[
X_k+\lm W\in (\partial \|\cdot\|_*)^{-1}(Y_k),\; \mbox{i.e.},\; Y_k\in \partial\|X_k+\lm W\|_*. 
\] 
It follows that 
\[
-\frac{1}{\mu} \Phi^*(\Phi(X_k+\lm W)-B_k)=-\frac{1}{\mu} \Phi^*( \Phi(X_k-B_k))=Y_k\in \partial\|X_k+\lm W\|_*,
\]
which tells us that $X_k+\lm W$ is also an optimal solution of problem \eqref{NNM2}. As $X_k$ is the only solution of problem \eqref{NNM2}, we have $W=0$. This verifies condition \eqref{Rad}. 

According to formula \eqref{Inver}, we have 
\[
\R_+ \big[(\partial \|\cdot\|_*)^{-1}(Y_k)-X_k\big]=\R_+\left[\OU\begin{pmatrix}\mathbb{S}^p_+-\Sigma_k&0\\0&0\end{pmatrix}\OV^T\right]= \OU\begin{pmatrix}\mathbb{S}^p&0\\0&0\end{pmatrix}\OV^T \; \mbox{with}\;\Sigma_k=\begin{pmatrix}\Sigma&0\\0&\frac{1}{k}\Id_{p-r}\end{pmatrix},
\]
where the last equality holds due to $\Sigma_k\in {\rm int}\, \mathbb{S}^p_+$. Combining this with \eqref{Rad} tells us that Strong Sufficient Condition \eqref{SSC2} is satisfied. 
\end{proof}

 Although the statement of Theorem~\ref{LipNu} is quite similar to that of Theorem~\ref{Lip12}, it is still surprising to see  the Lipschitz continuity of $S$ follows from its single-valuedness. The proof of Theorem~\ref{LipNu} relies on  the necessery condition~\eqref{Rad} for  solution uniqueness of problem \eqref{NNM2} at $X_k$; see, e.g.,    \cite[Lemma~3.1]{LPB21} and \cite[Theorem~3.1]{FNP24} for the idea of using the radial cone to characterize solution uniqueness of convex optimization problems. 

Next we show that the solution mapping $S$ is single-valued and Lipschitz continuous around $(\Bar B,\bar\mu)$ with $S(\Bar B,\bar \mu)=\OX$ if and only if $S(\Bar B,\bar \mu)=\OX$ under the following well-known Nondegeneracy Condition  (a.k.a Strict Complementarity Condition) widely used for nuclear norm minimization problem \eqref{NNM} \cite{CR13,ZS17}
\begin{eqnarray}\label{ND}
    \OY\eqdef-\frac{1}{\bar \mu}\Phi^*(\Phi \OX-\Bar B)\in {\rm ri}\, \partial \|\OX\|_*=\left\{U\begin{pmatrix} \Id_r&0\\ 0 & W\end{pmatrix}V^T|\; \|W\|< 1\right\} \;\; \mbox{for any}\;\; (U,V)\in \mathcal{O}(\OX). 
\end{eqnarray}
Hence the Nondegeneracy Condition \eqref{ND} means that $p=p(\OY)={\rm rank}\, \OX=r.$ The aforementioned equivalence has some roots from
 the recent one  \cite[Corollary~3.10]{FNP24}, which shows the equivalence between tilt stability and solution uniqueness under Nondegeneracy Condition \eqref{ND}. Here we provide a direct proof.

\begin{Corollary}\label{UniLip} Suppose that the Nondegeneracy Condition \eqref{ND} is satisfied. Then the solution mapping $S$ is single-valued and Lipschitz continuous around $(\Bar B,\bar\mu)$ with $S(\Bar B,\bar \mu)=\OX$ if and only if $\OX$ is a unique optimal solution of problem ${\mathcal P}(\Bar B,\bar \mu)$, i.e., $\OX= S(\Bar B,\bar \mu)$. 
\end{Corollary}
\begin{proof} We only need to prove that if the Nondegeneracy Condition \eqref{ND} is satisfied and  $\OX$ is an optimal solution of problem $P(\Bar B,\bar \mu)$, then $S$ is single-valued and Lipschitz continuous around $(\Bar B,\bar\mu)$. Indeed, if $\OX=S(\Bar B,\mu)$, it is similar to \eqref{Rad} that
\begin{equation}\label{Rad2}
\Ker \Phi\cap \R_+ \big[(\partial \|\cdot\|_*)^{-1}(\OY)-\OX\big]=\{0\}.
\end{equation}
As $\OY\in {\rm ri}\, \|\OX\|_*$, we have $r={\rm rank}\, \OX=p(\OY)=p$. It follows from \eqref{Inver} that  
\[
\R_+ \big[(\partial \|\cdot\|_*)^{-1}(\OY)-\OX\big]=\R_+\left[\OU\begin{pmatrix}\mathbb{S}^r &0\\0 &0\end{pmatrix}\OV^T
\right]=\OU\begin{pmatrix}\mathbb{S}^r &0\\0 &0\end{pmatrix}\OV^T
,
\]
where $(\OU,\OV)$ comes from the SVD of $\OX$ and $\OY$ in \eqref{SVD}. Condition \eqref{Rad2} is  exactly \eqref{SSC2}. By Theorem~\ref{LipNu}, the solution mapping $S$ is single-valued and Lipschitz continuous around $(\Bar B,\mu)$. 
\end{proof}

\begin{Remark}[Perturbations on $\Phi$]{\rm Similarly to Remark~\ref{PPhi}, when the operator $\Phi \in \mathscr{L}(\R^{n_1\times n_2},\R^m)$ is also a parameter around some $\Bar \Phi \in \mathscr{L}(\R^{n_1\times n_2},\R^m)$, the mapping $S(\Phi,B,\mu)$ defined in \eqref{SPb} is single-valued and Lipschitz continuous around $(\Bar\Phi,\Bar B,\bar \mu)$ for $\OX=S(\Bar\Phi,\Bar B,\bar \mu)$ if and only if 
\begin{equation*}
 \Ker \Bar\Phi\cap \OU\begin{pmatrix}\mathbb{S}^p&0\\0&0\end{pmatrix}\OV^T=\{0\} \;\;\mbox{with}\;\; \OY=-\frac{1}{\bar \mu}\Bar\Phi^*(\Bar\Phi\, \OX-\Bar B),\; p=p(\OY)\; \mbox{and }\; (\OU,\OV)\in \mathcal{O}(\OX)\cap\mathcal{O}(\OY).
\end{equation*}
Moreover, if the Nondegenerate Condition is satisfied in this format, i.e., $\OY\in {\rm ri}\, \partial \|\OX\|_*$,  the above Lipschitz stability of $S(\Phi,B,\mu)$ is also valid if and only if $\OX=S(\Bar\Phi,\Bar B,\bar \mu)$.  \hfill$\triangle$}
\end{Remark}
\section{Conclusion} 

In this paper we obtain necessary and sufficient conditions for Lipschitz stability of the solution mapping $S$ of parametric problem~\eqref{CP}. Along the way, full geometric characterizations of  tilt stability are also obtained for more general problem \eqref{CPS} without computing the generalized Hessian on $g$. For some particular important regularizers like the $\ell_1/\ell_2$ norm or the nuclear norm, we show that the Lipschitz stability is automatic whenever the solution mapping is single-valued around the point in question. Computing the Lipschitz modulus in this case is a reachable project for future research. As our analysis relies on the quadratic growth condition  of $g$ for the set $\mathcal{M}$ and  the relative approximations of  $\partial g$, it is natural to ask if   other important classes such as {\em twice epi-differentiable functions} \cite{RW98}, {\em spectral functions} \cite{CDZ17}, and {\em composite function} \cite{BS00} satisfy these two conditions under some circumstances. Moreover, the set $\mathcal{M}$ is only obtained case-by-case in this paper. Constructing $\mathcal{M}$, analyzing it,  and evaluating the set $\mathfrak{T}_\mathcal{M}$ in \eqref{T} for more general functions $g$ are what we want to proceed next.  Moreover, comparing the set $\mathfrak{T}_\mathcal{M}$ with the outer liming graphical derivative $D^\sharp \partial g^*(\oy, \ox)(0)$ \cite{GO22} as well as the generalized Hessian $D^*\partial g^*(\oy, \ox)(0)$ \cite{PR98,M92} is significant to connect different approaches in this direction for future research.

Another interesting  topic for ongoing research is about sensitivity analysis of the solution mapping $S$ under our Strong Sufficient Condition. Some impressive results about sensitivity analysis of parametric problem~\eqref{CP} on different classes of regularizers were  studied on \cite{BBH22,BLPS21,BPS22,VDFPD17}. Example~\ref{Non} shows that Strong Sufficient Condition does not guarantee the solution mapping $S$ to be differentiable, but whether $S$ is  directional differentiable or possesses the {\em conservative Jacobian} \cite{BLPS21, BPS22} is an open question, especially for group Lasso and nuclear norm minimization problems.

\vspace{0.2in}

\noindent{\bf Acknowledgments.} The author is deeply grateful to Tim Hoheisel and Defeng Sun for many fruitful discussions about this research topic. He is  indebted to Ebrahim Sarabi for helpful remarks on characterizing tilt stability via second subderivatives. The author is also thankful to both anonymous referees for their careful readings, constructive suggestions, as well as deep remarks on other works 
that helped improve the original presentation significantly.


\begin{thebibliography}{99}
\small 
\bibitem{AG08} F. J. Arag\'on Artacho and M. H.  Geoffroy: Characterizations of metric regularity of subdifferentials, \textit{J. Convex Anal.}, {\bf 15}, 365--380, 2008.

\bibitem{BBH22} A. Berk, S.  Brugiapaglia, and T. Hoheisel: Lasso reloaded: a variational analysis perspective with applications to compressed sensing, {\em SIAM J. Math. Data Sci.}, {\bf 5}, 1102--1129,  2023.

\bibitem{BGM19} M. Benko, H. Gfrerer, and B. S. Mordukhovich: Characterizations of tilt-stable minimizers in second-order cone programming, {\em  SIAM J.  Optim.}, {\bf 29}, 3100--3130, 2019.

\bibitem{BLPS21} J. Bolte, T. Le, E. Pauwels, and A. Silveti-Falls: Nonsmooth implicit differentiation for machine learning and optimization, {\em Advances in Neural Information Processing Systems} (NeurIPS 2021). 34:13537--13549.

\bibitem{BPS22} J. Bolte, E. Pauwels, and A. Silveti-Falls: Differentiating nonsmooth solutions to parametric monotone inclusion problems,  {\em SIAM J. Optim.}, {\bf 34}, 71--97,  2024.

 \bibitem{BNPS17} J. Bolte, T. P. Nguyen, J. Peypouquet, and B. W. Suter: From error bounds to the complexity of first-order descent methods for convex functions, \textit{Math. Program.}, {\bf 165}, 471--507, 2017.



\bibitem{BS00}
J.~F. Bonnans and A.~Shapiro:
\newblock {\em Perturbation analysis of optimization problems}.
\newblock Springer Series in Operations Research. Springer-Verlag, New York,
  2000.

 \bibitem{BS87} A. Blake and A. Zisserman: {\em Visual reconstruction}, MIT press, 1987.

\bibitem{CHN18}
N.~H. Chieu, L.~V. Hien, and T.~T.~A. Nghia.
\newblock Characterization of tilt stability via subgradient graphical
  derivative with applications to nonlinear programming.
\newblock {\em  SIAM J. Optim.}, {\bf 28}, 2246--2273, 2018.



\bibitem{CHNT21}N.~H. Chieu, L.~V. Hien,  T.~T.~A. Nghia, and H. A. Tuan: Quadratic growth and strong metric subregularity of the subdifferential via subgradient graphical derivative,  \textit{SIAM J. Optim.}, 31, 545--568, 2021.

  \bibitem{CRPW12} V. Chandrasekaran, B. Recht, P. A Parrilo, and A. S. Willsky: The convex geometry of linear inverse problems, \textit{Found Comput Math}, \textbf{12}, 805--849,  2012.

 \bibitem{CP10} E.  Cand\`es
and Y. Plan: Matrix completion with noise, \textit{Proceeding of the IEEE}, \textbf{98}, 925--936, 2010.

\bibitem{CR09} E. Cand\`es and B. Recht: Exact matrix completion via convex optimization \textit{Found. Comput. Math.}, \textbf{9}, 717--772, 2009.

\bibitem{CR13}  E. Cand\`es and B. Recht: Simple bounds for recovering low-complexity models, \textit{Math. Program.},  \textbf{141}, 577--589, 2013.

\bibitem{C16} Y. Cui: Large scale composite optimization problems with coupled objective
functions: theory, algorithms and applications, \textit{PhD thesis}, National University of Singapore, 2016.

\bibitem{CD19} Y. Cui and C. Ding: Nonsmooth composite matrix optimization: strong regularity, constraint nondegeneracy and beyond,	arXiv:1907.13253

\bibitem{CDZ17} Y. Cui, C. Ding, and X.  Zhao: Quadratic growth conditions for convex matrix optimization problems associated with spectral functions, \textit{SIAM J. Optim.}, {\bf 27}, 2332--2355, 2017.

\bibitem{DL13} D. Drusvyatskiy and A. S. Lewis: Tilt stability, uniform quadratic growth, and strong
metric regularity of the subdifferential, {\em  SIAM J. Optim.} {\bf 23},  256--267, 2014.

\bibitem{DMN14} D. Drusvyatskiy, B. S. Mordukhovich, and T. T. A.  Nghia: Second order growth, tilt stability, and metric regularity of
the subdifferential,  \textit{J. Convex Anal.}, {\bf 21}, 1165--1192, 2014.


\bibitem{DI15}
D.~Drusvyatskiy and A.~D. Ioffe:
\newblock Quadratic growth and critical point stability of semi-algebraic
  functions.
\newblock {\em Math. Program.}, {\bf 153}, 635--653, 2015.


\bibitem{DR09} A. L. Dontchev and R. T. Rockafellar: \textit{Implicit Functions and Solution Mappings}, \textit{A View from Variational Analysis}, Springer, Dordrecht, 2009.





 \bibitem{FNT21} J. Fadili, T. T. A. Nghia, and T. T. T. Tran: Sharp, strong and unique minimizers for low complexity robust recovery, \textit{IMA Inf. Inference}, \textbf{12}, 1461--1513, 2023.

\bibitem{FNP23} J. Fadili, T. T. A. Nghia, and D. N. Phan: Geometric characterizations for strong minima with applications to nuclear norm minimization problems, arXiv:2308.09224v1, 2023. 

\bibitem{FNP24} J. Fadili, T. T. A. Nghia, and D. N. Phan: Solution uniqueness of convex optimization problems via the radial cone, {\em J. Optim. Theory Appl.},  arXiv:2401.10346, 2024 (to appear).


 \bibitem{KMP22} P. D. Khanh, B. S. Mordukhovich, and V. T. Phat:  A generalized Newton method for subgradient systems,  {\em Math. Oper. Res.}, {\bf 48}, 1811--2382, 2023.

 \bibitem{G24} H. Gfrerer: On a globally convergent semismooth$^*$ Newton method in nonsmooth nonconvex optimzation, 	arXiv:2403.10142, 2024. 

 \bibitem{GO21} H. Gfrerer and J. V.  Outrata:  On a semismooth$^*$ Newton method for solving generalized equations, {\em SIAM J. Optim.}, {\bf 31}, 489–517, 2021.

\bibitem{GO16} H. Gfrerer and J. V. Outrata: On Lipschitzian properties of implicit multifunctions,
\newblock {\em  SIAM J. Optim.}, {\bf 26}, 2160--2189, 2016.

\bibitem{GO22} H. Gfrerer and J. V. Outrata:  On (local) analysis of multifunctions via subspaces contained in graphs of generalized derivatives, {\em J.  Math. Anal.  Appl.}, {\bf 508}, 125895, 2022.



 \bibitem{GM15}
H.~Gfrerer and B.~S. Mordukhovich:
\newblock Complete characterizations of tilt stability in nonlinear programming
  under weakest qualification conditions.
\newblock {\em  SIAM J. Optim.}, {\bf 25}, 2081--2119, 2015.

 \bibitem{HG11} M. Hebiri and S. van de Geer: The smooth-lasso and other $\ell_1+\ell_2$-penalized methods, {\em 
Electronic J. Stat.}, {\bf 5},1184--1226, 2011.

\bibitem{LZ09} H. Liu and J. Zhang: Estimation consistency of the group lasso and its applications,  {\em Journal
of Machine Learning Research}, {\bf 5}, 376--383, 2009. 

 \bibitem{LZ13} A. S. Lewis and S. Zhang: Partial smoothness, tilt stability, and generalized Hessians, {\em SIAM J. Optim.}, {\bf 23}, 74--94, 2013. 

\bibitem{LPB21} Y. Liu, S. Pan, and S. Bi: Isolated calmness of solution mappings and exact recovery conditions for nuclear norm optimization problems, {\em Optimization}, {\bf 70}, 481--510, 2021. 

\bibitem{OM23} W. Ouyang and A. Milzarek: Variational properties of decomposable functions. Part II: Strong second-order theory,  arXiv:2311.07276, 2023.

 \bibitem{PR98}
R.~A. Poliquin and R.~T. Rockafellar:
\newblock Tilt stability of a local minimum.
\newblock {\em SIAM J. Optim.}, {\bf 8}, 
287--299, 1998.

 \bibitem{M92} B. S. Mordukhovich: Sensitivity analysis in nonsmooth optimization, in {\em Theoretical Aspects of Industrial Design} (D. A. Field and V. Komkov, eds.), SIAM Proc. Applied Math., Vol. 58, pp. 32--46, SIAM, Philadelphia, PA, 1992.


\bibitem{M1} B. S. Mordukhovich: {\em Variational Analysis and Generalized Differentiation}, Springer, 2006.

\bibitem{MNR15} B. S. Mordukhovich, T. T. A. Nghia, and R. T. Rockafellar: Full stability in finite-dimensional optimization, {\em Math. Oper. Res.}, {\bf 40}, 226--252, 2015. 

\bibitem{MS21} B. S. Mordukhovich and E. Sarabi: Generalized Newton algorithms for tilt-stable minimizers in nonsmooth optimization, {\em SIAM J. Optim.}, {\bf 31}, 1184--1214, 2021. 

 \bibitem{MR12} B. S. Mordukhovich and R. T. Rockafellar:
\newblock Second-order subdifferential calculus with applications to tilt stability in optimization, {\em SIAM J. Optim.},
{\bf 22}, 953--986, 2012.

\bibitem{MN15} B. S. Mordukhovich and T. T. A. Nghia:
\newblock Second-order characterizations of tilt stability with applications to nonlinear programming,
\newblock {\em Math. Program.} {\bf 149}, 83--104, 2015.

\bibitem{MY12} J. Mairal and B. Yu: Complexity analysis of the lasso regularization path. In Proceedings of the 29th {\em International Conference on Machine Learning}, 1835--1842, 2012.

\bibitem{R80} S. M. Robinson: Strongly regular generalized equations, \textit{Math. Oper.
Res.}, {\bf 5}, 43--62, 1980.

\bibitem{R81} S. M. Robinson: Some continuity properties of polyhedral multifunctions, \textit{Math.
Program. Study}, {\bf 14}, 206--214, 1981.




\bibitem{R70} R. T. Rockafellar: {\em Convex Analysis}, Princeton University Press, Princeton, New Jersey, 1970.

\bibitem{R23} R. T. Rockafellar:Augmented Lagrangians and hidden convexity in sufficient conditions for local optimality, {\em Math. Program.} {\bf 198}, 159--194, 2023.

\bibitem{RW98} R. T. Rockafellar and  R. J-B. Wets: {\em Variational Analysis}, Springer, Berlin, 1998.



\bibitem{S06} D. Sun: The strong second-order sufficient condition and constraint nondegeneracity in nonlinear semidefinite programming and their
applications, {\em Math. Oper. Res.}, {\bf 31}, 761--776, 2006.


\bibitem{T13} R. J. Tibshirani: The Lasso problem and uniqueness, {\em Electron. J. Stat.}, {\bf 7}, 1456--1490, 2013.

\bibitem{VDFPD17} S. Vaiter, C. Deledalle, J. Fadili, G. Peyr\'e, and C. Dossal: The degrees of freedom of partly smooth regularizers, {\em Ann. Inst. Statist. Math.}, {\bf 69}, 791--832, 2017.

\bibitem{YL06} M. Yuan  and  Y. Lin: Model selection and estimation in regression with grouped variables. \textit{Journal of the Royal Statistical Society}, Series B, \textbf{68}, 49--67, 2006.



\bibitem{W92} G. A. Watson: Characterization of the subdifferential of some matrix norms, \textit{Linear Algebra Appl.}, {\bf 170} (1992), 33--45.

\bibitem{W19} M. J. Wainwright: {\em High-Dimensional Statistics. A Non-Asymptotic Viewpoint}, Cambridge University Press, Cambridge, 2019.  

\bibitem{ZT95}
R. Zhang and J. Treiman:
 Upper-Lipschitz multifunctions and inverse subdifferentials, \textit{\em Nonlinear Anal.}, {\bf 24} (1995), 273--286.

 \bibitem{ZH05} H. Zou and T. Hastie: Regularization and variable selection via the elastic net. {\em J. Royal Stat. Soc.  B} (Statistical Methodology), {\bf 67} (2005), 301–-320.

\bibitem{ZS17} Z. Zhou and A. M-C.  So: A unified approach to error bounds for structured convex optimization, \textit{Math. Program.}, {\bf 165} (2017), 
689--728.






\end{thebibliography}
\end{document}